\documentclass[a4paper]{article}
\usepackage{mystyle}
\usepackage{tikz}
\usetikzlibrary{arrows.meta, positioning}

\newcommand{\GNEP}{\mathrm{GNEP}(\{(W^\nu_0(\cdot),\theta^\nu)\}_{\nu=1}^N)}
\newcommand{\hatGNEP}{\widehat{\mathrm{GNEP}}(\{(\hat{W}^\nu_0(\cdot),\theta^\nu)\}_{\nu=1}^N)}
\newcommand{\NEP}{\mathrm{NEP}_\rho(\{(W^\nu,\Theta^\nu_\mu(\cdot;\rho))\}_{\nu=1}^N)}
\newcommand*{\hatRGNEP}[1][\delta]{\widehat{\mathrm{RGNEP}}_{#1}(\{(\hat{W}^\nu_{#1^\nu}(\cdot),\theta^\nu)\}_{\nu=1}^N)}
\newcommand{\graph}{\mathrm{gph}\ }

\newcommand{\dist}{\mathrm{dist}}
\renewcommand{\Re}{\mathbb{R}}

\definecolor{mygreen}{RGB}{0,150,60}
\definecolor{horired}{RGB}{207,0,0}

\title{%
    A Regularized Nikaido--Isoda Function Approach to Multi-Leader--Follower Games%
    \thanks{This work was supported by 
            Japan Society for the Promotion of Science,
            and Grant-in-Aid for Early-Career Scientists (JP26K21172),
            Grant-in-Aid for Scientific Research (C) (JP25K15008, JP25K15002), and
            Grant-in-Aid for Scientific Research (B) (JP25K03082).}
}
\author{
    Atsushi Hori\thanks{Graduate School of Engineering, Nagoya Institute of Technology, Aichi, Japan.}
    \and
    Takayuki Okuno\thanks{Faculty of Science and Technology, Seikei University, Tokyo, Japan.}
    \and 
    Ellen H. Fukuda\thanks{Graduate School of Informatics, Kyoto University, Kyoto \mbox{606--8501}, Japan.}
}
\date{\today}

\begin{document}
\maketitle


\begin{abstract}
A multi-leader--follower game (MLFG) is a hierarchical noncooperative game in which leaders compete at the upper level while taking into account the followers' best responses at the lower level.
A typical approach to solving the MLFG reformulates it as an equilibrium problem with equilibrium constraints (EPECs) by replacing the lower-level game with its KKT conditions.
Another approach, when each follower's response is unique,
is to reformulate the MLFG as a Nash equilibrium problem by substituting these response functions into each leader's problem.
However, both reformulations may lack scalability since higher-order derivatives may be required  when solving the resulting problems.

In this paper, we propose a new reformulation of the MLFG by exploiting a regularized Nikaido--Isoda function
and approximating the MLFG by a single-level differentiable Nash equilibrium problem
with a penalty parameter.
The proposed reformulation neither requires derivative information on the followers' game nor assumes convexity of each follower's problem; hence, it can handle a broader class of MLFGs.
Under global subanalyticity, we analyze 
the mathematical relationship between 
equilibria of the original MLFG and the proposed reformulation.

\end{abstract}


\section{Introduction}

The multi-leader--follower game (MLFG) is a hierarchical
noncooperative model
that extends the classical Stackelberg game.
In this framework, multiple leaders first take into account the response of (possibly multiple) followers, who subsequently
engage in a Nash equilibrium problem among themselves.
Such a game was introduced in Heinrich F.~von Stackelberg's 
seminal work in 1934, prior to Nash's equilibrium 
theory of noncooperative games in 1950~\cite{Nash1950}.

In recent decades, the MLFG has gained significant attention across various
fields due to its ability to model hierarchical decision-making in decentralized
systems.
Some  applications include strategic bidding in electricity markets~\cite{Allevi2018,Lei2023,Leyffer2010}, resource management in cloud and edge computing~\cite{Kim2018,Lyu2022,Xiong2019,Xu2021,Zhang201675s},
and market design~\cite{Ashraf2023,Kumar2022,Vazifeh2021}.
For instance, in electricity markets studied by
Leyffer and Munson~\cite{Leyffer2010} and Pang and Fukushima~\cite{Pang2005n},
generators (leaders) set prices 
strategically before the market clears, while other market participants
(followers) respond by adjusting their demand and supply accordingly.
For surveys of the MLFG, see Hu and Fukushima~\cite{Hu2015} and
Aussel and Svensson~\cite{Aussel2020}.

Despite their applicability, solving MLFGs remains challenging.
In fact, MLFGs are more complex than standard bilevel optimization
because they necessitate finding a Nash equilibrium at the upper level,
where each leader's strategy is constrained by the equilibrium of the
lower-level game.
Historically, two main approaches have been considered in the MLFG literature.

One approach is to solve the MLFG by reformulating it into an
equilibrium problem with equilibrium constraints (EPEC).
By replacing the followers' Nash equilibrium problems with 
Karush--Kuhn--Tucker (KKT) conditions, each leader's problem is transformed
into a mathematical problem with equilibrium constraints (MPEC).
Leyffer and Munson~\cite{Leyffer2010} proposed an MPEC-based method by 
considering a first-order necessary condition of all leaders' MPECs,
leading to systems of complementarity conditions.
Then, they are solved with algorithms given by Su~\cite{Su2004}.
Later, Hori and Fukushima~\cite{Hori2019} proposed a diagonal method for the 
resultant EPEC using a penalty method adapted to MLFGs and
utilizing a sequential penalty method developed by Huang et al.~\cite{Huang2006}.

Another approach introduces a response function of followers' game,
assuming the uniqueness of the Nash equilibrium.
By plugging the unique response (e.g., denoted by $y(x)$, where $x$ is 
a tuple of leaders' strategies) into each leader's problem,
the MLFG is reduced to an ordinary single-level Nash equilibrium problem.
However, because the response function $y(x)$ is generally nonsmooth,
smoothing schemes are required.
Hu and Fukushima~\cite{Hu2011} first proposed a response-based method 
where a single follower solves an equality-constrained quadratic optimization
problem with an analytical response.
Herty et al.~\cite{Herty2022} extended this framework to followers with linear inequality
constraints and proposed a smoothing scheme to obtain an approximate leader--follower Nash equilibrium with gradient-based methods.
%
%
These studies assumed that the followers' response $y(x)$ is analytically obtained.
In contrast to these approaches, Hori et al.~\cite{Hori2024} proposed a general framework that does not assume specific structures for players' objective functions and constraints.

Despite those developments, existing studies have some drawbacks. 
For the EPEC-based method, reformulating the followers' game
as a KKT system requires first-order derivatives.
As a consequence, solving the resulting EPEC often requires higher-order derivatives of the followers' functions.
Moreover, when the lower-level problem is not convex, MPEC and EPEC reformulations 
are usually not equivalent to the original problem; see~\cite{Dempe2007,Dutta2006}.
As for response-based approaches, they often assume the uniqueness of the followers' response, 
which is practically the same as assuming strong convexity or monotonicity in the followers' noncooperative game.

In this paper, motivated by those limitations, we propose a new reformulation of
the MLFG by utilizing the Nikaido--Isoda function for the followers' game.
Our approach neither needs any derivative information nor assumes strong convexity in the followers' problems.
We reformulate the MLFG into a differentiable single-level
Nash equilibrium problem with penalty terms, and then analyze the 
relationship between the resultant Nash game and the original MLFG.
In fact, similar approaches have recently attracted the attention of researchers in
bilevel optimization, where several algorithms~\cite{Chen202567l,Lu2024,Yao2024,Yao2024gy6}
have been proposed, which support our approach.

Our contributions are summarized as follows:
\begin{itemize}
    \item 
    We propose a differentiable single-level reformulation of the MLFG,
	which does not use any derivative information with respect to
	each follower's problem, by  using the regularized Nikaido--Isoda function.

    \item We establish the relationship between the 
	leader--follower equilibrium of the original MLFG and Nash equilibrium 
	of the resultant single-level Nash equilibrium problem
	(Theorem~\ref{thm:relations.NE});

	\item
    We also establish the relationship between the \emph{variational
	equilibria}, which means the stationarity condition,
	of the original MLFG and the single-level Nash equilibrium problem
	(Theorem~\ref{thm:VE.rel}).
\end{itemize}

This paper is organized as follows:
Section~\ref{sec:preliminaries} introduces mathematical notions and 
Nash equilibrium problems to provide fundamental tools for analysis.
Section~\ref{sec:reformulation} proposes a new level-reduction technique
for the MLFG with regularized Nikaido--Isoda function and approximates
it by a certain penalized problem.
Section~\ref{sec:relations} analyzes the relationship between the penalized
problem and the original MLFG.
Section~\ref{sec:conclusion} concludes this paper.

Throughout this work, we use the following notations:
Given $m$ vectors $x^1,\dots,x^m$, where $x^i\in\Re^{n_i}$ for $i=1,\dots,m$,
the concatenated vector $x\coloneq ((x^1)^\top,\dots,(x^m)^\top)^\top$ with $^\top$ being the transpose, is simply written as $x=(x^1,\dots,x^m)$.
To emphasize the $i$th subvector $x^i\in\Re^{n_i}$, $x$ is sometimes denoted as $x=(x^i,x^{-i})$,
where $x^{-i}\coloneq (x^1,\dots,x^{i-1},x^{i+1},\dots,x^m)$ represents the tuple of all subvectors 
except $x^i$.
For the real-valued function $f\colon\Re^{n+m}\to\Re$, $\nabla_y f(x,y)$ 
denotes the partial gradient of $f$ 
with respect to the second variable $y$.
For the vector-valued function $F\colon\Re^n\to\Re^m$, the transposed Jacobi
matrix $\mathcal{J}F(x)^\top$ is denoted by $\nabla F(x)\in\Re^{n\times m}$;
we simply call it the Jacobian or Jacobi matrix of $F$.
Similarly, for $F\colon\Re^{n+m}\to\Re^l$, the partial (transposed) Jacobi matrix
with respect to $x$ is denoted by $\nabla_x F(x,y)$.
Also, the Euclidean norm and inner product are written as $\|\cdot\|$ and $\langle \cdot,\cdot \rangle$, respectively.
We denote by $\mathbb{B}$ the open unit ball centered at 0, and let $\mathbb{B}(x,\varepsilon)\coloneq\{y\in\Re^n\mid \|y-x\|<\varepsilon\}
=x+\varepsilon\mathbb{B}$ for a vector $x\in\Re^n$ and 
positive real number $\varepsilon>0$.
Let $\bar{\mathbb{B}}$ and $\bar{\mathbb{B}}(x,\varepsilon)$ be the closure of the open balls,
i.e., the closed balls.
We denote the normal cone to a convex set $\Sigma \subset \Re^n$ at $x\in\Sigma$ as $\mathcal{N}_{\Sigma}(x) \coloneq \{s\in\Re^n \mid \langle s, y-x \rangle \le 0\; \forall y \in \Sigma\}$.
The Euclidean distance from $x\in\Re^n$ to $S\subset\Re^n$ is denoted by $\dist(x,S)$.

\section{Preliminaries}\label{sec:preliminaries}

In this section, we provide some concepts regarding convex analysis
and recall the (generalized) Nash equilibrium problem.

\subsection{Notations and fundamental properties}

The graphs $\graph f$ and $\graph T$ of the real-valued function 
$f\colon\Re^n\to\Re$ and set-valued mapping
$T\colon\Re^n\rightrightarrows\Re^m$ are defined as follows:
$$
\graph f\coloneq\{(x,\lambda)\in\Re^n\times\Re\mid \lambda=f(x)\},
$$
and
$$
\graph T\coloneq\{(x,\lambda)\in\Re^n\times\Re^m\mid \lambda\in T(x)\},
$$
respectively.
The function $f$ is $L$-smooth if there exists $L>0$ such that
$$
\|\nabla f(x)-\nabla f(x')\|\le L\|x-x'\|\quad\forall x,x'\in \Re^n,
$$
and $f$ is said to be \emph{$\mu$-weakly convex} with $\mu > 0$ if $f+\mu/2\|\cdot\|^2$ is convex.

The following lemmas are well-known results in nonlinear optimization;
we omit those proofs.

\begin{lemma}\label{lem:weakly.convex}
	Suppose that $f$ is $L$-smooth. 
	Then, $f$ is $L$-weakly convex, i.e., $f+L/2\|\cdot\|^2$ is convex.
\end{lemma}

\begin{lemma}\label{lem:smooth.lipschitz}
	Suppose that $f$ is $L$-smooth on a compact set $X$.
	Then, $\nabla f$ is bounded and $f$ is Lipschitz continuous.
\end{lemma}

\begin{lemma}\label{lem:C2.imply.smooth}
	Let $X\subset\Re^n$ be a compact set.
	Let $f\colon\Re^n\to\Re$ be twice continuously differentiable over an
	open set $U$ that contains $X$.
	Then, there exists $L>0$ such that $f$ is $L$-smooth.
\end{lemma}
Next, we introduce the subanalyticity along with other relevant concepts.
\begin{definition}[Subanalyticity~{\cite[Definition~2.1]{Bolte2007}}]\label{def:subanalyticity}
	\mbox{\\}
	\begin{enumerate}
		\item A subset $A$ of $\Re^n$ is called \emph{semianalytic} if each point of $\Re^n$ admits a neighborhood $V$
			for which $A\cap V$ assumes the following form:
			$$
			\bigcup_{i=1}^p \bigcap_{j=1}^q \{x\in V\mid f_{ij}(x)=0, g_{ij}(x)>0\},
			$$
			where the functions $f_{ij},g_{ij}\colon V\to\Re$ are real-analytic for all $i\in\{1,\dots,p\}$ and
			$j\in\{1,\dots,q\}$.
		\item The set $A$ is called \emph{subanalytic} if each point of $\Re^n$ admits a neighborhood $V$ such that
			$$
			A\cap V=\{x\in \Re^n\mid (x,y)\in B\},
			$$
			where $B$ is a bounded semianalytic subset of $\Re^n\times\Re^m$ for some $m\ge 1$.
		\item Given $m,n\in\mathbb{N}$, a function $f\colon\Re^n\to\Re\cup\{\infty\}$ (respectively,
			a point-to-set operator $T\colon\Re^n\rightrightarrows\Re^m$) is called 
			\emph{subanalytic} if its graph is subanalytic subset of $\Re^n\times\Re$
			(respectively, of $\Re^n\times\Re^m$).
	\end{enumerate}
\end{definition}

\begin{definition}[{Global subanalyticity~\cite{Dries1996}}]\label{def:global.subanalyticity}
	Let $x\in\Re^n$, and
	define the function
	$$
	\Phi_n(x) \coloneq 
	\begin{pmatrix}
		\frac{x_1}{1+x_1^2},\dots,\frac{x_n}{1+x_n^2}
	\end{pmatrix}.
	$$
	\begin{enumerate}
		\item A subset $A$ of $\Re^n$ is called \emph{globally subanalytic} if its image under $\Phi_n$ is a 
			subanalytic subset of $\Re^n$;
		\item A function $f\colon\Re^n\to\Re\cup\{\infty\}$ (respectively, a point-to-set operator
			$T\colon\Re^n\rightrightarrows\Re^m$) is called \emph{globally subanalytic} if 
			$\graph f$ (respectively, $\graph T$) is a globally subanalytic subset of $\Re^n\times\Re$ (respectively, of $\Re^n\times\Re^m$).
	\end{enumerate}
\end{definition}

Loosely speaking, 
a (globally) subanalytic function is a function described as a combination of locally analytic functions.
Subanalyticity may be generally satisfied by many widely-used objective functions~\cite{Bolte2007}.
They exhibit a ``tame'' geometry, thus stability under basic operations, and desirable properties for
optimization~\cite{Chen202567l}.
One of them is the H\"olderian error bound defined later.

\begin{lemma}[{Bolte et al.~\cite{Bolte2007}}]\label{lem:globally.subanalytic}
	Globally subanalytic sets are subanalytic, and any bounded subanalytic set is globally subanalytic.
\end{lemma}

\begin{lemma}[{Dries and Miller~\cite{Dries1996}}]\label{lem:image.subanalytic}
	The image or the preimage of a globally subanalytic set by a globally subanalytic function (respectively,
	globally subanalytic multivalued operator) is globally subanalytic.
\end{lemma}


\begin{definition}[H\"olderian error bound]\label{def:Holderian.EB}
	For the function $f\colon X\to\Re$, let $X^*\coloneq\argmin_{x\in X} f(x)$.
	Then, we say that $f$ satisfies the $(\xi,\eta)$-H\"olderian error bound (HEB) if
	$$
	f(x)-\min_{x\in X} f(x)\ge \xi\cdot \dist(x,X^*)^\eta\qquad \forall x\in X,
	$$
	where $\xi,\eta>0$.
\end{definition}

If the HEB condition holds for a function, it measures the distance in a certain sense between the solution set
to the optimization problem by the gap of objective function values.

The following lemma ensures that if a solution set and objective function are globally subanalytic,
the objective function satisfies the HEB for some~$(\xi,\eta)$.

\begin{lemma}[{\L}ojasiewicz factorization lemma~\cite{Bierstone1988}]\label{lem:factorization.lemma}
	If the solution set $X^*\coloneq\argmin_{x\in X} f(x)$ is globally subanalytic, and $f$ is continuous and globally subanalytic,
	then $f$ satisfies H\"olderian error bound.
\end{lemma}

\begin{lemma}[{Kosiba~\cite[Lemma~4.19]{Kosiba2025}}]
	\label{lem:gl.subanalytic.value.fn}
	Let $X$ and $Y$ be bounded and globally subanalytic subsets in $\Re^n$ and $\Re^m$, respectively.
	Suppose that $f\colon X\times Y\to\Re$ is a bounded subanalytic function.
	Then, the value-function defined by $\phi(x)\coloneq\min_{y\in Y} f(x,y)$ is bounded and subanalytic,
	thus globally subanalytic.
\end{lemma}



\subsection{Nash equilibrium problem}

Consider an $N$-person noncooperative game~\cite{Nash1950}.
Player labeled with $\nu\in\{1,\dots,N\}$ solves
\begin{align}\label{prob:player}
	\min_{x^\nu\in\Re^{n_\nu}}\quad \theta^\nu(x^\nu,x^{-\nu})\qquad
	\text{s.t}\quad x^\nu\in X^\nu,
\end{align}
where $x^\nu\in\Re^{n_\nu}$ denotes a strategy for player $\nu$, and 
$\theta^\nu\colon\Re^n\to\Re$ and $X^\nu\subset\Re^{n_\nu}$ denote their cost function and 
strategy space, respectively.
Here, $n\coloneq n_1+\dots+n_N$.

\begin{definition}[Nash equilibrium]\label{def:Nash.eq}
	We say a tuple of strategies $x^*\coloneq (x^{*,1},\dots,x^{*,N})\in
	 X\coloneq X^1\times\dots\times X^N$ is a \emph{Nash equilibrium} (NE)
	of the noncooperative game in which player $\nu\in\{1,\dots,N\}$ solves~\eqref{prob:player} if it satisfies
	the following inequality for all $\nu$:
	$$
	\theta^\nu(x^{*,\nu},x^{*,-\nu})\le \theta^\nu(x^\nu,x^{*,-\nu})\qquad \forall x^\nu\in X^\nu.
	$$
\end{definition}
The NE means that nobody can reduce their cost function by only changing their strategy unilaterally.
We refer to the problem to find the NE of the noncooperative game
as a \emph{Nash equilibrium problem} (NEP).
The following proposition ensures the existence of NE for the NEP. 
For a proof, we refer the reader to the literature, e.g., Aubin~\cite{Aubin1979}.

\begin{proposition}[Existence of Nash equilibrium]\label{prop:existence.NE}
	For all $\nu\in\{1,\dots,N\}$, suppose that $X^\nu\subset\Re^{n_\nu}$ is nonempty, convex, and compact.
	Suppose also that $\theta^\nu$ is continuous, and 
	$\theta^\nu(\cdot,x^{-\nu})$ is convex for arbitrary fixed $x^{-\nu}$.
    Then, the NEP has an NE.
\end{proposition}

Now we introduce the \emph{Nikaido--Isoda} (NI) \emph{function}
(sometimes called the \emph{Ky-Fan function}) for the NEP as follows:
\[
	\Psi(x,z) \coloneq \sum_{\nu=1}^N
	\left[\theta^\nu(x^\nu,x^{-\nu})-\theta^\nu(z^\nu,x^{-\nu})\right],
\]
We also introduce the \emph{value-function} or \emph{gap function}
for the NEP defined by
\[
	V(x) \coloneq \sup_{z\in X} \Psi(x,z).
\]
It is easy to see that $x^*\in X$ is an NE if and only if $V(x^*)=0$.

Next, we introduce the \emph{regularized} NI function for the NEP as follows:
\[
	\Psi_\mu(x,z) \coloneq \sum_{\nu=1}^N
	[ \theta^\nu(x^\nu,x^{-\nu})-\theta^\nu(z^\nu,x^{-\nu}) ]
	-\frac1{2\mu}\|x-z\|^2,
\]
where $\mu>0$ is referred to as a \emph{regularization parameter},
and its value-function is $V_\mu(x) \coloneq \sup_{z\in X}\Psi_\mu(x,z)$;
we refer to $V_\mu$ as a \emph{regularized value-function}.
The regularized value-function is often used in NEPs and their generalization~\cite{Facchinei2010,Gurkan2009,Heusinger2009} which are introduced later.
In fact, those ideas are based on the Moreau--Yosida regularization or Moreau envelope of $\theta^\nu$.
When $\theta^\nu(\cdot,x^{-\nu})$ is convex for each $\nu$,
$V_{\mu}$ is differentiable,  whereas $V$ is not in general.

If each player's problem is convex, $x^*\in X$ is an NE if and only if $V_\mu(x^*)=0$.
Otherwise, this equivalency is not true as illustrated in the following example.
\begin{example}\label{ex:nonconvex.NEP}
	Consider a two-person NEP, in which
	\[
		\theta^1(x_1,x_2)\coloneq -x_1^2x_2,\ 
		\theta^2(x_1,x_2)\coloneq x_1x_2^2,\ 
		X^1=X^2\coloneq [0,1].
	\]
	For the regularization parameter $\mu=1/2$,  it holds that
	\[
		\begin{aligned}
		V_\mu(0,1) &= \theta^1(0,1)+\theta^2(0,1)-
		\min_{z_1\in[0,1]}\{\theta^1(z_1,1)+(z_1-0)^2\}\\
		&\quad-\min_{z_2\in[0,1]}\{\theta^2(0,z_2)+(z_2-1)^2\} = 0 + 0 - 0 - 0 = 0.
		\end{aligned}
	\]
	On the other hand,
    $\theta^1(\epsilon,1)=-\epsilon^2<0=\theta^1(0,1)$ holds
	for any $\epsilon>0$, and thus, $(0,1)$ is not an NE for the NEP.
\end{example}

The generalized Nash equilibrium problem (GNEP) extends the NEP by allowing each player's strategy set to depend on the other players' strategies. More precisely, the constraint of each player $\nu$ in the GNEP is represented as $x^\nu \in X^\nu(x^{-\nu})$. Such constraints are called \emph{coupling constraints}.
In this paper, we consider a particular class of GNEP with $K\subset\Re^n$,  where each player solves 
\begin{align}\label{prob:player.GNEP}
	\min_{x^\nu\in\Re^{n_\nu}}\quad \theta^\nu(x^\nu,x^{-\nu}) \qquad
	\text{s.t.} \quad (x^\nu,x^{-\nu})\in K. 
\end{align}
In this GNEP, 
$X^\nu(x^{-\nu})=\{x^\nu\in\Re^{n_\nu}\mid (x^\nu,x^{-\nu})\in K\}$ for each $\nu$, 
which means that  all the players share the same coupling constraint.

\begin{definition}[Generalized Nash equilibrium]\label{def:GNE}
	A vector $x^*\in K$ is referred to as a \emph{generalized Nash equilibrium} (GNE) of the GNEP if 
	$x^{*,\nu}$ solves~\eqref{prob:player.GNEP} with $x^{-\nu}=x^{*,-\nu}$ for all $\nu\in\{1,\dots,N\}$.
\end{definition}


\section{Level-reduction technique for MLFGs and its penalty minimization
with regularized Nikaido--Isoda functions}
\label{sec:reformulation}
In this section, we begin by formally formulating the MLFG considered in this paper. 
Subsequently, we reformulate the MLFG as a GNEP
and propose a penalty minimization approach using the regularized NI function to solve the obtained GNEP. 

\subsection{Problem formulation}
In this paper, we consider the MLFG comprising $N$ leaders and $M$ followers.
In this MLFG, NEPs are induced on both the leaders' side and the followers' side. 
We denote  a leader $\nu$'s strategy and a tuple of their rivals' strategies  
by $x^{\nu}\in \Re^{n_{\nu}}$ and $x^{-\nu}\in \Re^{n-n_{\nu}}$ with $n \coloneq n_1+\dots+n_N$, respectively.  
We also denote 
a follower $\omega$'s strategy and a tuple of their rivals' strategies 
by $y^{\omega}\in \Re^{m_{\omega}}$ and $y^{-\omega}\in \Re^{m-m_{\omega}}$ with $m \coloneq m_1+\dots+m_M$, respectively. 
Each leader indexed by $\nu\in\{1,\dots,N\}$ solves the following problem:
\begin{align}\label{prob:leader}
	\min_{x^\nu\in\Re^{n_\nu}}\quad \theta^\nu(x^\nu,x^{-\nu},y)\qquad
	\text{s.t.}\quad x^\nu\in X^\nu,
\end{align}
where $\theta^\nu\colon\Re^{n+m}\to\Re$.
On the other hand, each follower indexed by $\omega\in\{1,\dots,M\}$ solves 
\begin{align}\label{prob:follower}
	\min_{y^\omega\in\Re^{m_\omega}} \quad \gamma^\omega(x,y^\omega,y^{-\omega}) \qquad
	\text{s.t.}\quad y^\omega\in Y^\omega,
\end{align}
where $\gamma^{\omega}\colon\Re^{n+m}\to\Re$.
Hereinafter, let
$$
X \coloneq X^1\times X^2 \times \cdots\times X^N\subset \Re^n,\
Y \coloneq Y^1\times Y^2\times \cdots \times Y^M\subset \Re^m.
$$

\subsection{Existing level-reduction approaches}

In this subsection, suppose that $Y^\omega$ is given by 
\[
Y^\omega = \{y^\omega\in\Re^{m_\omega}\mid g^\omega(y^\omega)\le 0\}
\footnote{Although we can also consider equality constraints, we omit them for simplicity.}, 
\]
where $g^\omega\colon\Re^{m_\omega}\to\Re^{p_\omega}$.
Moreover, assume that each follower's problem is convex, i.e.,~\eqref{prob:follower} is convex
for arbitrary fixed $x$ and $y^{-\omega}$, and suitable constraint qualifications hold.
The Nash game is equivalently reformulated as the following KKT systems:
\begin{align}\label{KKT.system}
\begin{cases}
\nabla_{y^\omega} \gamma^\omega(x,y^\omega,y^{-\omega}) + \nabla_{y^\omega} g^\omega(y^\omega)\lambda^\omega = 0,\\
0\le\lambda^\omega \perp g^\omega(y^\omega)\le 0
\end{cases}
\quad\forall \omega\in \{1,\dots,M\},
\end{align}
where $\lambda^\omega\in\Re^{p_\omega}$ is the Lagrange multiplier for the inequality constraint
$g^\omega(y^\omega)\le 0$, and $\lambda^\omega\perp g^\omega(y^\omega)$ means that $\langle \lambda^\omega,g^\omega(y^\omega)\rangle=0$.

Incorporating~\eqref{KKT.system} into the constraint
of each leader's problem~\eqref{prob:leader}, we have the following mathematical problem
with equilibrium constraints (MPEC): Leader $\nu$ solves
\begin{align}\label{prob:MPEC}
\min_{x^\nu,y,\lambda} \theta^\nu(x^\nu,x^{-\nu},y) \qquad \text{s.t. } x^\nu\in X^\nu,~\eqref{KKT.system}.
\end{align}
The NEP in which leader~$\nu$ solves~\eqref{prob:MPEC} is referred to  as an \emph{equilibrium problem with equilibrium constraints} (EPEC).

A typical approach to address the EPEC associated with the MLFG is based on methods for MPECs~\cite{Franci2025,Hori2019,Leyffer2010}.
Those methods are gradient-based, which, in turn, require Hessians $\nabla^2_y \gamma^\omega$ and $\nabla^2_{y^\omega}g^\omega$ or higher-order derivatives.

Specifically, if the followers' best response is unique for every $x\in X$, there exists a mapping of the response $y\colon x \mapsto y(x)$.
Then, plugging $y(x)$ into each leader's problem, the MLFG is reduced to an ordinary NEP as follows: Leader $\nu$ solves
\begin{align}\label{prob:response}
    \min_{x^\nu\in X^\nu} \quad& \theta^\nu(x^\nu,x^{-\nu},y(x^\nu,x^{-\nu})).
\end{align}
Since $y(x)$ is not necessarily smooth, 
Hori et al.~\cite{Hori2024} proposed a smoothing scheme. 
To obtain a smooth approximated response, they utilized a 
smoothing approximation of a nonlinear equation induced from the KKT systems~\eqref{KKT.system}, 
in which the complementarity conditions are replaced by the smoothing Fischer--Burmeister functions~\cite{Facchinei1999}.
Then, the reduced problem~\eqref{prob:response} is approximated by the following (differentiable) problem:
\begin{align}\label{prob:response.approx}
    \min_{x^\nu\in X^\nu} \quad& \theta^\nu(x^\nu,x^{-\nu},y_\varepsilon(x^\nu,x^{-\nu})),
\end{align}
where $y_\varepsilon(x)$ is the $y$-part of a solution to the smoothed approximated nonlinear equations for~\eqref{KKT.system},
and $\varepsilon>0$ is a smoothing parameter.
When applying first-order methods to~\eqref{prob:response.approx}, one needs to compute $\nabla y_\varepsilon (x)$,
which involves computing  $\nabla^2_{y} \gamma^\omega$ and $\nabla^2_{y^\omega} g^\omega$
via the implicit function theorem; see the appendix for details.
However, this can be computationally demanding in large-scale settings.

Consequently, both EPEC and the best response approaches require computing the Hessians of the followers' defining functions $\gamma^\omega$ and $g^\omega$.
This motivates us to develop an alternative reformulation for the MLFG.

\subsection{Level-reduction of the MLFG into a single-level GNEP}

Using the Nikaido--Isoda function, the set of NE for the followers'
NEP is represented by
\[
S(x) \coloneq \{y\in Y\mid h(x,y) \coloneq  \sup_{z\in Y} \Psi(x,y,z)\le 0\},
\]
where 
\[
	\Psi(x,y,z) \coloneq \sum_{\omega=1}^M
	[\gamma^\omega(x,y^\omega,y^{-\omega})-
	\gamma^\omega(x,z^\omega,y^{-\omega})].
\]
We refer to the function $h$ as a value-function
for the followers' NEP.
By letting 
$$\phi^\omega(x,y^{-\omega}) \coloneq 
\min_{z^\omega\in Y^\omega} \gamma^\omega(x,z^\omega,y^{-\omega}),$$
we have
\[
	h(x,y)=\sum_{\omega=1}^M [\gamma^\omega(x,y^\omega,y^{-\omega})
	-\phi^\omega(x,y^{-\omega})].
\]

Incorporating the condition $y\in S(x)$ into each leader's problem 
reduces the MLFG to
the following single-level GNEP with shared coupling constraints:
\begin{align}\label{prob:reduced.before}
\min_{x^\nu,\hat{y}^\nu} \quad \theta^\nu(x^\nu,x^{-\nu},{\hat{y}}^\nu) \qquad
\text{s.t.} \quad x^\nu\in X^\nu,\ \hat{y}^\nu\in S(x^\nu,x^{-\nu}).
\end{align}
Note that $y^{\omega}\in \Re^{m_{\omega}}$ denotes the  follower $\omega$'s strategy, 
whereas  
$\hat{y}^{\nu}\in \Re^m$ denotes a tuple of all the followers' strategies observed by leader $\nu$.
Let 
\begin{align*}
\hat{y} \coloneq (\hat{y}^1,\dots,\hat{y}^N)\in\Re^{Nm},\ w^\nu \coloneq (x^\nu,\hat{y}^\nu)\in\Re^{n_\nu+m},
\end{align*}
and 
$$
W^\nu_0(x^{-\nu}) \coloneq \{w^\nu\in\Re^{n_\nu+m}\mid x^\nu\in X^\nu, \hat{y}^\nu\in S(x^\nu,x^{-\nu})\}.
$$
With these notations,~\eqref{prob:reduced.before} is simply rewritten as
\begin{align}\label{prob:reduced}
	\min_{w^\nu\in\Re^{n_\nu+m}} \quad \theta^\nu(w^\nu,x^{-\nu})\qquad
	\text{s.t.}\quad w^\nu\in W^\nu_0(x^{-\nu}).
\end{align}

The reduced noncooperative game in which each leader $\nu$ solves~\eqref{prob:reduced} is
denoted as $\GNEP$.
Now, we  introduce an equilibrium concept for the MLFG via  $\GNEP$ as follows:
\begin{definition}[Leader--follower Nash equilibrium]\label{def:LFNE}
	A tuple of strategies $(x^*,y^*)\in X\times S(x^*)$ is
	a \emph{leader--follower Nash equilibrium} (LFNE)  of the MLFG
    if for all
	$\nu\in\{1,\dots,N\}$, $w^{*,\nu} \coloneq (x^{*,\nu},\hat{y}^{*,\nu})$ 
	solves~\eqref{prob:reduced} with 
	$x^{-\nu}=x^{*,-\nu}$, i.e.,
	$$
	w^{*,\nu}\in\argmin_{w^\nu}\{\theta^\nu(w^\nu,x^{*,-\nu})\mid w^\nu\in W^\nu_0(x^{*,-\nu})\}.
	$$
\end{definition}
Note that the term \emph{leader--follower} NE is used because
the GNE of the problem $\GNEP$ is a solution of the MLFG, and vice versa.
The definition originates from an `optimistic' LFNE introduced by Hu and Fukushima~\cite{Hu2015}.





Note further that 
the above reformulation does not require any derivative information on followers' functions, 
whereas the existing level-reduction approaches require it as explained in the previous subsection.
\subsection{Differentiable penalty minimization problem}
In this section, we introduce a new penalized problem for solving \eqref{prob:reduced}. 
To this end, we first define 
a `regularized' Nikaido--Isoda function for the followers' NE as follows:
$$
\Psi_\mu(x,y,z) \coloneq \sum_{\omega=1}^M
\{\gamma^\omega(x,y^\omega,y^{-\omega})-\gamma^\omega(x,z^\omega,y^{-\omega})\}-
\frac{1}{2\mu}\|y-z\|^2.
$$
Accordingly, the `regularized' value-function for the followers' NEP is defined~as
$$
h_\mu(x,y) \coloneq \max_{z\in Y} \Psi_\mu(x,y,z).
$$
Letting 
$$
\phi^\omega_\mu(x,y^\omega,y^{-\omega}) \coloneq \min_{z^\omega\in Y^\omega}\gamma^\omega(x,z^\omega,y^{-\omega})+
\frac{1}{2\mu}\|y^\omega-z^\omega\|^2,
$$
it holds that
\begin{equation}
h_\mu(x,y)=\sum_{\omega=1}^M \gamma^\omega(x,y^\omega,y^{-\omega})-
\phi^\omega_\mu(x,y^\omega,y^{-\omega}).
\label{eq:hmu}
\end{equation}

Next, we make the following assumptions on the MLFG.
Notice that convexity is assumed for neither the leaders' functions $\theta^{\nu}$ nor the followers' functions $\gamma^{\omega}$.
\begin{assumption}\label{asmp}
	For all $\nu\in\{1,\dots,N\}$ and $\omega\in\{1,\dots,M\}$, the following conditions hold:
	\begin{enumerate}[ref={\theassumption.\arabic{enumi}}]
		\item $X^\nu$ is a convex compact subset of $\Re^{n_\nu}$;
		\item \label{asmp:compact.follower} $Y^\omega$ is a convex compact subset of $\Re^{m_\omega}$;
		\item \label{asmp:smooth.leader}
		      $\theta^\nu$ is twice continuously differentiable on an open bounded
		      set containing $X\times Y$;
		\item \label{asmp:smooth.follower}
		      $\gamma^\omega$ is twice continuously differentiable on an open bounded
		      set containing $X\times Y$.
	\end{enumerate}
\end{assumption}

\begin{assumption}\label{asmp:subanalyticity}
    For all $\nu\in\{1,\dots,N\}$ and $\omega\in\{1,\dots,M\}$, the following conditions hold:
    \begin{enumerate}[ref={\theassumption.\arabic{enumi}}]
		\item $\theta^\nu$ is subanalytic on $X\times Y$;
		\item $\gamma^\omega$ is subanalytic on $X\times Y$;
    \end{enumerate}
\end{assumption}

Under these conditions, we can ensure the functions 
$\theta^{\nu},\gamma^{\omega}$, and $\phi^{\omega}_{\mu}$
are tractable in the following sense:
\begin{proposition}\label{prop:new}
	Under Assumptions~\ref{asmp} and~\ref{asmp:subanalyticity}, the following properties hold:
\begin{enumerate}
\item The functions $\theta^\nu$ and $\gamma^\omega$ are $L$- and $\ell$-smooth
for some $L,\ell>0$; 
\item\label{item:wc} $\gamma^\omega$ is $\ell$-weakly convex;
\item $\phi^\omega_\mu$ is continuously differentiable for all $\mu<\ell^{-1}$.
\end{enumerate}
\end{proposition}
\begin{proof}
Since 
$\theta^\nu$ and $\gamma^\omega$ are twice continuously differentiable by Assumption~\ref{asmp},
Lemma~\ref{lem:C2.imply.smooth} guarantees that 
$\theta^\nu$ and $\gamma^\omega$ are $L_\nu$- and $\ell_\omega$-smooth, respectively. 
Hence, for sufficiently large $L,\ell>0$, we ensure 
$\theta^\nu$ and $\gamma^\omega$ are $L$- and $\ell$-smooth, respectively.

The $\ell$-smoothness of $\gamma^\omega$, together with
Lemma~\ref{lem:weakly.convex}, shows that $\gamma^\omega$ is $\ell$-weakly convex.
Since $\gamma^\omega(x,\cdot,y^{-\omega})$ is also $\ell$-weakly convex,
\begin{align*}
  & \gamma^\omega(x,z^\omega,y^{-\omega})+
  \frac{1}{2\mu}\|y^\omega-z^\omega\|^2 \\
  = \: & \gamma^\omega(x,z^\omega,y^{-\omega}) + \frac{\ell}{2} \|z^\omega\|^2
  + \frac{1}{2} \left( \frac{1}{\mu}-\ell \right) \|z^\omega\|^2
  + \frac{1}{2\mu} \|y^\omega\|^2 - \frac{1}{\mu} \langle y^\omega, z^\omega \rangle
\end{align*}
is strongly convex on $z^\omega$ when $1/\mu-\ell>0$.
Then, the solution to the inner minimization of $\phi^\omega_\mu$ is
always unique for all $\mu<\ell^{-1}$.
Thus, together with Assumptions~\ref{asmp:compact.follower} and~\ref{asmp:smooth.follower},
Danskin's theorem (see, e.g.,~\cite[Remark~4.14]{Bonnans2000}) ensures that $\phi^\omega_\mu$
is continuously differentiable.
\end{proof}


In the literature, $h_\mu(x,y)$ has been extensively investigated for convex NEPs, i.e., the case where $\gamma^\omega(x,\cdot,y^{-\omega})$
is convex for any fixed $x$ and $y^{-\omega}$, and $Y^\omega$ is convex.
These results can be extended to
nonconvex NEPs as follows:

\begin{proposition}\label{prop:NIfunc}
	Under Assumption~\ref{asmp}, the following statements are valid:
	For arbitrary $x\in X$,
	\begin{enumerate}
        \item\label{prop:NIfunc:item1} $h_\mu(x,\cdot)$ is nonnegative on $Y$;
		\item Suppose that $S(x)\neq\emptyset$.
			If $y^*\in Y$ is an NE of the followers' NEP, i.e., $y^*\in S(x)$,
			then $h_\mu(x,y^*)=0$ holds.
			Moreover, if for all $\omega\in\{1,\dots,M\}$, $\gamma^\omega(x,\cdot,y^{-\omega})$ is convex for arbitrary $y^{-\omega}$,
			the converse is also valid;
		\item For every $y\in Y$ and $\mu<\ell^{-1}$,
        there exists a unique $z^*(x,y)$ such that
			\begin{align}\label{eq:best.response}
			   z^*(x,y)=\argmax_{z\in Y} \Psi_\mu(x,y,z),
			\end{align}
			and $z^*(x,y)$ is continuous in $(x,y)$.
	\end{enumerate}
\end{proposition}
\begin{proof}
	The proof is readily shown based on~\cite[Theorem~3.3]{Heusinger2009}.
\end{proof}

Hereinafter, assume that $\mu<\ell^{-1}$.
From Danskin's theorem, we have
\[
	\nabla h_\mu(x,y) = \sum_{\omega=1}^M
	\begin{bmatrix}
		\nabla_x \gamma^\omega(x,y) - \nabla_x \gamma^\omega(x,z^{*,\omega}(x,y),y^{-\omega}) \\
		\nabla_y \gamma^\omega(x,y) - \varpi^\omega(x,z^{*,\omega}(x,y),y^{-\omega})
	\end{bmatrix},
\]
where $z^{*,\omega}(x,y)=\argmin_{z^\omega\in Y^\omega} \{\gamma^\omega(x,z^\omega,y^{-\omega})+1/(2\mu)\|z^\omega-y^\omega\|^2\}$ and
\[
	\varpi^\omega(x,z^{*,\omega}(x,y),y^{-\omega}) \coloneq 
	\begin{bmatrix}
		\nabla_{y^1} \gamma^\omega(x,z^{*,\omega}(x,y),y^{-\omega}) \\ 
		\vdots \\
		\nabla_{y^{\omega-1}} \gamma^\omega(x,z^{*,\omega}(x,y),y^{-\omega}) \\
		\frac1\mu(y^\omega-z^{*,\omega}(x,y)) \\
		\nabla_{y^{\omega+1}} \gamma^\omega(x,z^{*,\omega}(x,y),y^{-\omega}) \\
		\vdots \\
		\nabla_{y^M} \gamma^\omega(x,z^{*,\omega}(x,y),y^{-\omega})
	\end{bmatrix}.
\]

Define 
\[
	\hat{S}(x)  \coloneq  \{y\in Y \mid h_\mu(x,y)\le 0\}.
\]
By Proposition~\ref{prop:NIfunc}, it holds that
\begin{equation}
	S(x) \subset \hat{S}(x),\label{eq:sxshs}
\end{equation}
where the equality holds
if $\gamma^\omega(x,\cdot,y^{-\omega})$ is convex.
The following proposition characterizes $\hat{S}(x)$
by the first-order conditions of the followers' problems. 
The proof is similar to the result by \cite{Liu2024} for bilevel optimization.
\begin{proposition}\label{prop:first-order.condition}
	Suppose that Assumptions~\ref{asmp:compact.follower}
	and~\ref{asmp:smooth.follower} hold.
	Then, 
	\[
		\hat{S}(x) = T(x) \coloneq\left\{
			y\in Y\ \middle|\ 
			0\in \nabla_{y^\omega}\gamma^\omega(x,y^\omega,y^{-\omega})+\mathcal{N}_{Y^\omega}(y^\omega),\ \omega=1,\dots,M
		\right\},
	\]
    recalling that
    $\mathcal{N}_{Y^\omega}(y^\omega)$ is the normal cone to $Y^{\omega}$ at $y^{\omega}$.
\end{proposition}
\begin{proof}
First, we show  $\hat{S}(x)\subset T(x)$.
Choose $y\in \hat{S}(x)$ arbitrarily. 
From $h_\mu(x,y)\le 0$ and \eqref{eq:hmu}, we have 
    \[  
		\begin{aligned}
             &\sum_{\omega=1}^M \gamma^\omega(x,y^\omega,y^{-\omega})\le \sum_{\omega=1}^M\phi^\omega_\mu(x,y^\omega,y^{-\omega})\\
			=&\sum_{\omega=1}^M \min_{z^\omega\in Y^\omega} [\gamma^\omega(x,z^\omega,y^{-\omega})+\frac1{2\mu}\|y^\omega-z^\omega\|^2]
			\le \sum_{\omega=1}^M \gamma^\omega(x,y^\omega,y^{-\omega}),
		\end{aligned}
	\]
\noindent and thus $\sum_{\omega=1}^M [\gamma^\omega(x,y^\omega,y^{-\omega}) - \phi^\omega_\mu(x,y^\omega,y^{-\omega})]=0$,
    which together with 
	 $\gamma^\omega(x,y^\omega,y^{-\omega})-\phi^\omega_\mu(x,y^\omega,y^{-\omega})\ge 0$ for all $\omega$
	from the definition implies 
	\[
		\gamma^\omega(x,z^\omega,y^{-\omega})-\phi^\omega_\mu(x,z^\omega,y^{-\omega}) = 0\quad \forall\omega\in\{1,\dots,M\}.
	\]
    Therefore, we have 
	$y^\omega\in\argmin_{z^\omega\in Y^\omega}\gamma^\omega(x,z^\omega,y^{-\omega})+\frac1{2\mu}\|y^\omega-z^\omega\|^2$ for all $\omega$,
    and hence we ensure that $y \in T(x)$.  

	Next, we show $\hat{S}(x)\supset T(x)$.
    Choose $y\in T(x)$ arbitrarily.
    From the second assertion of Proposition~\ref{prop:new},
   the function $\gamma^\omega(x,z^\omega,y^{-\omega})+\frac1{2\mu}\|y^\omega-z^\omega\|^2$ is convex with respect to
    $z^\omega$ on the compact convex set $Y^\omega$ under $\mu<\ell^{-1}$.
    Then, it follows from $y\in T(x)$ that for each $\omega$, we have
	\[
		y^\omega \in\argmin_{z^\omega\in Y^\omega}
		\gamma^\omega(x,z^\omega,y^{-\omega})+\frac1{2\mu}\|y^\omega-z^\omega\|^2,
	\]
which
yields
$\gamma^\omega(x,y^\omega,y^{-\omega})-\phi^\omega_\mu(x,y^\omega,y^{-\omega})=0$ for all $\omega$.
Therefore, 
$h_\mu(x,y)= 0$ holds and thus $y\in \hat{S}(x)$.
The proof is complete.
\end{proof}

Define the following problem by replacing the followers' Nash equilibrium condition 
$y\in S(x)$ in~\eqref{prob:reduced} with $y\in\hat{S}(x)$:
\begin{align}\label{prob:reduced.relaxed}
	\min_{w^\nu \in \Re^{n_\nu+m}} \quad \theta^\nu(w^\nu,x^{-\nu}) \qquad
	\text{s.t.}  \quad w^\nu=(x^\nu,\hat{y}^\nu)\in\hat{W}^\nu_0(x^{-\nu}),
\end{align}
where 
\begin{align*}
\hat{W}_0^\nu(x^{-\nu})
\coloneq& \{w^\nu\in\Re^{n_\nu+m} \mid x^\nu\in X^\nu, \hat{y}^\nu\in\hat{S}(x^\nu,x^{-\nu})\}\\
=&
\{w^\nu\in\Re^{n_\nu+m} \mid x^\nu\in X^\nu, \hat{y}^\nu\in Y, h_{\mu}(x^{\nu},x^{-\nu},\hat{y}^\nu)\le 0\}.
\end{align*}
We denote by $\hatGNEP$
the game 
in which leader $\nu$ solves~\eqref{prob:reduced.relaxed}.
$\hatGNEP$ is a relaxation of $\GNEP$ due to~ \eqref{eq:sxshs}.
If the followers' NEP is convex, i.e., the objective function $\gamma^\omega(x,\cdot,y^{-\omega})$ is
convex, 
$\hat{S}(x)=S(x)$ holds and hence this relaxation is tight; otherwise, the concept of the equilibrium 
differs from each other.

\begin{definition}[Normalized leader--follower Nash equilibrium]\label{def:normal.LFNE}
	A tuple of strategies $(x^*,y^*)\in X\times\hat{S}(x^*)$ is a 
    \emph{normalized leader--follower Nash equilibrium} (NoLFNE) of the MLFG
    if for all $\nu\in\{1,\dots,N\}$,
	$w^{*,\nu}=(x^{*,\nu},\hat{y}^{*,\nu})$ solves~\eqref{prob:reduced.relaxed} with fixed $x^{-\nu}=x^{*,-\nu}$,
	i.e.,
	\[
		w^{*,\nu}\in\argmin_{w^\nu}\{\theta^\nu(w^\nu,x^{*,-\nu})\mid w^\nu\in\hat{W}^\nu_0(x^{*,-\nu})\}.
	\]
\end{definition}

The following assertion holds obviously.

\begin{theorem}\label{thm:LFNE.NoLFNE}
	Suppose that the followers' NEP is convex for any tuples of 
	leaders' strategies, i.e., for all $\omega\in\{1,\dots,M\}$,
	$\gamma^\omega(x,\cdot,y^{-\omega})$ is convex for arbitrary
	fixed $x$ and $y^{-\omega}$.
	A point is an LFNE if and only if it is a NoLFNE. 
\end{theorem}

Despite the regularization of the value-function $h(x,y)$ (called $h_\mu(x,y)$),
$\hatGNEP$ is still difficult to solve in practice since
\begin{itemize}
	\item the constraint of each leader's problem still
	depends on $x^{-\nu}$;
	\item ordinary constraint qualifications, such as the
	Mangasarian--Fromovitz constraint qualification, do not hold in general~\cite[Proposition~3.2]{Ye1995}, which will be discussed
	in detail in Section~\ref{ssec:VE}.
\end{itemize}

Now we consider a partially penalized problem of~\eqref{prob:reduced.relaxed}
as follows:
\begin{align}\label{prob:penalty}
	\begin{aligned}
		\min_{w^\nu} &\quad \Theta^\nu_\mu(w^\nu,x^{-\nu};\rho) \coloneq 
		\theta^\nu(w^\nu,x^{-\nu})+\rho h_\mu(w^\nu,x^{-\nu}) \\
		\text{s.t.} &\quad w^\nu \in W^\nu  \coloneq  X^\nu\times Y,
	\end{aligned}
\end{align}
where $\rho>0$ is a penalty parameter. 
The penalization utilizes the property of $h_\mu$ that
\begin{itemize}
	\item $h_\mu(x,y)\ge 0$ for all $x\in X$ and $y\in Y$;
	\item $h_\mu(x,y)\le 0$ implies that $y\in\hat{S}(x)$
	(respectively, $y\in S(x)$ if $\gamma^\omega(x,\cdot,y^{-\omega})$ is
		convex).
\end{itemize}

Note that~\eqref{prob:penalty} for $\nu=1,\dots,N$, are regarded as a single-level NEP; thus,
the game in which leader $\nu\in\{1,\dots,N\}$ solves~\eqref{prob:penalty}
is referred to as $\NEP$ or \emph{$\rho$-penalized NEP},
and $\hatGNEP$ is referred to as \emph{constrained GNEP}.

In order to analyze some properties of $\NEP$ in the next section,
we introduce the following relaxed $\hatGNEP$, in which
each leader's reduced problem~\eqref{prob:reduced.relaxed} is relaxed as
\begin{align}\label{prob:relaxed.problem}
	\min_{w^\nu}\quad \theta^\nu(w^\nu,x^{-\nu})
	\qquad \text{s.t.}\quad
	w^\nu\in\hat{W}^\nu_{\delta^\nu}(x^{-\nu}),
\end{align}
where
\[
	\hat{W}^\nu_{\delta^\nu}(x^{-\nu}) \coloneq 
	\{w^\nu\in W^\nu\mid h_\mu(w^\nu,x^{-\nu})\le\delta^\nu\}.
\]
For $\delta\coloneq(\delta^1,\dots,\delta^N)$, we denote by $\hatRGNEP$ the noncooperative game in which leader
$\nu\in\{1,\dots,N\}$ solves~\eqref{prob:relaxed.problem}, which is called $\delta$-\emph{relaxed GNEP}.

\section{Relationships between \texorpdfstring{$\rho$}{rho}-penalized NEP and 
\texorpdfstring{$\delta$}{delta}-relaxed GNEP}
\label{sec:relations}

In this section, we study the relationships between $\NEP$ and $\hatRGNEP$ 
with particular emphasis on their solution sets.
We summarize the relationships between the MLFG and these NEPs in Figure~\ref{fig:rel} and Table~\ref{tab:rel}, provided in the end of this section.


\subsection{Approximate Nash equilibria of $\rho$-penalized NEP and relaxed GNEP}


First,  we introduce an approximate NE of $\NEP$ and an approximate GNE of $\hatRGNEP$.

\begin{definition}[Approximate Nash equilibrium for $\rho$-penalized NEP]
	\label{def:approx.NE}
	For $\varepsilon>0$, a point $w^* \coloneq (w^{*,1},\dots,w^{*,N})\in W\coloneq W^1\times\dots\times W^N$ is said 
	to be an \emph{$\varepsilon$-Nash equilibrium}
	or $\varepsilon$-NE of $\rho$-penalized NEP ($\NEP$) if 
	the following inequality holds for all $\nu\in\{1,\dots,N\}$:
	$$
	\Theta^\nu_\mu(w^{*,\nu},x^{*,-\nu};\rho)\le\Theta^\nu_\mu(w^\nu,x^{*,-\nu};\rho)+\varepsilon\quad\forall w^\nu\in W^\nu.
	$$
\end{definition}

\begin{definition}[Approximate generalized Nash equilibrium for $\delta$-re\-laxed GNEP]
	\label{def:approx.GNE}
	For $\varepsilon>0$ and $\delta\in\Re^N_{++}$, $w^*\in \hat{W}_\delta(x^*) \coloneq 
	\prod_{\nu=1}^N \hat{W}^\nu_{\delta^\nu}(x^{*,-\nu})$ is
	said to be an $\varepsilon$-GNE of $\delta$-relaxed GNEP
	($\hatRGNEP$) if the following inequality holds for all $\nu\in\{1,\dots,N\}$:
	$$
	\theta^\nu(w^{*,\nu},x^{*,-\nu})\le\theta^\nu(w^\nu,x^{*,-\nu})+\varepsilon
	\quad\forall w^\nu\in \hat{W}^\nu_{\delta^\nu}(x^{*,-\nu}).
	$$
\end{definition}

We analyze the relationship between $\NEP$ and \\ $\hatRGNEP$ in terms of approximate (G)NE.
Some auxiliary lemmas are given before the main assertions.

\begin{lemma}\label{lem:lipschitz.hmu}
	Under Assumption~\ref{asmp}, $h_\mu$ is Lipschitz continuous on $X\times Y$.
\end{lemma}
\begin{proof}
The claim readily follows from the continuous differentiability of $h_{\mu}$
together with boundedness of $X$ and $Y$.
\end{proof}

\begin{lemma}\label{lem:HEB.hmu}
	Suppose that Assumptions~\ref{asmp} and~\ref{asmp:subanalyticity} hold.
	Then, $h_\mu$ is globally subanalytic and satisfies $(\xi,\eta)$-HEB on $X\times Y$ with some $\xi>0$ and $\eta>0$.
\end{lemma}
\begin{proof}
	By Assumptions~\ref{asmp} and~\ref{asmp:subanalyticity}, $\gamma^\omega$ is bounded and subanalytic on the 
	compact set $X\times Y$.
	Utilizing Lemma~\ref{lem:gl.subanalytic.value.fn}
	leads that $\phi^\omega_\mu$ is globally subanalytic.
	From the definition of $h_\mu$ and Lemma~\ref{lem:image.subanalytic}, 
	$h_\mu$ is subanalytic.
	In addition, by the Lipschitz continuity of $h_\mu$ from Lemma~\ref{lem:lipschitz.hmu},
	$h_\mu$ is bounded and subanalytic, thus globally subanalytic by Lemma~\ref{lem:globally.subanalytic}.
	Therefore, it follows from Lemmas~\ref{lem:factorization.lemma}
	and~\ref{lem:gl.subanalytic.value.fn} that
	$h_\mu$ satisfies $(\xi,\eta)$-HEB.
\end{proof}

The following lemma concerns the gap of the objective values $\Theta^\nu_\mu$
between a feasible point of $\hat{W}^\nu_0(x^{-\nu})$ and some
point $w^\nu\in W^\nu$ for a fixed $x^{-\nu}$.

\begin{lemma}\label{lem:gap.Theta}
Suppose that Assumption~\ref{asmp} holds.
Then, $\theta^\nu$ is $\bar{L}$-Lipschitz continuous
on $X\times Y$ with a constant $\bar{L}>0$, and
$h_\mu$ is globally subanalytic on $X\times Y$; hence, there exists $\xi>0$ and $\eta>0$ such that $h_\mu$ satisfies $(\xi,\eta)$-HEB.
Let $\hat{W}^\nu_0(x^{-\nu})\neq\emptyset$
and $u^\nu(w^\nu)\in\argmin_{u^\nu\in\hat{W}^\nu_0(x^{-\nu})}\|u^\nu-w^\nu\|$ for $w^\nu\in W^\nu$.
It holds that
\begin{align*}
&        \Theta^\nu_\mu(w^\nu,x^{-\nu};\rho)
-\Theta^\nu_\mu(u^\nu(w^\nu),x^{-\nu};\rho)\\
=&\theta^\nu(w^\nu,x^{-\nu})+ \rho h_\mu(w^\nu,x^{-\nu})-
\theta^\nu(u^\nu(w^\nu),x^{-\nu}) \\
\ge&  -\varepsilon_\rho, 
\end{align*}
where
\begin{equation*}
\varepsilon_\rho \coloneq 
\begin{cases}
\bar{L}\left(\frac{\bar{L}\xi}{\rho \eta}\right)^{\frac{1}{\eta-1}}(1-\frac{1}{\eta}), & \eta>1,\rho>0;\\
0, & \eta=1,\rho\ge {\xi} \bar{L},
\end{cases}
\end{equation*}
\end{lemma}
\begin{proof}
	The assertion can be easily derived by regarding $x^{-\nu}$ as a
	fixed parameter, as demonstrated in \cite[Lemma~D.2]{Chen202567l}.
\end{proof}

The following theorem states that if a point is an approximate $\varepsilon$-NE of $\NEP$, then
it is also that of $\hatRGNEP$ for some $\delta$ depending on the penalty parameter $\rho$.

\begin{theorem}\label{thm:relations.NE}
	For $\varepsilon,\rho>0$,
    let $w_\rho \in W$ be an $\varepsilon$-NE of the $\rho$-penalized NEP and assume $\hat{W}^\nu_0(x^{-\nu}_\rho)\neq\emptyset$ for all $\nu$.
	Then, there exists $\delta_\rho\coloneq(\delta^1_\rho,\dots,\delta^N_\rho)\in\Re^N_{++}$ such that $w_\rho$ is
	the $\varepsilon$-GNE of $\hatRGNEP[\delta_{\rho}]$; that is,
	$$
	\theta^\nu(w^\nu_\rho,x^{-\nu}_\rho)\le \theta^\nu(w^\nu,x^{-\nu}_\rho)+\varepsilon\quad\forall w^\nu\in \hat{W}^\nu_{\delta_\rho^\nu}(x^{-\nu}_\rho).
	$$
\end{theorem}
\begin{proof}
	By the assumption $\hat{W}^\nu_0(x^{-\nu}_\rho)\neq\emptyset$, we have
	\begin{align*}
		& \theta^\nu(w^\nu_\rho,x^{-\nu}_\rho)+\rho h_\mu(w^\nu_\rho,x^{-\nu}_\rho)\le \min_{w^\nu\in W^\nu}\theta^\nu(w^\nu,x^{-\nu}_\rho)+\rho h_\mu(w^\nu,x^{-\nu}_\rho)+\varepsilon \\
		\le & \min_{w^\nu\in \hat{W}^\nu_0(x^{-\nu}_\rho)}\theta^\nu(w^\nu,x^{-\nu}_\rho)+\rho h_\mu(w^\nu,x^{-\nu}_\rho) +\varepsilon \\
		\le & \theta^\nu(w^\nu,x^{-\nu}_\rho)+\varepsilon\quad\forall w^\nu\in \hat{W}^\nu_0(x^{-\nu}_\rho),
	\end{align*}
    where the second inequality follows from $\hat{W}^\nu_0(x^{-\nu}_\rho)\subset W^\nu$.
	It follows from Lemma~\ref{lem:gap.Theta} that for $\rho\neq\rho'>0$, we have
	$$
	\theta^\nu(w^\nu_\rho,x^{-\nu}_\rho)+\rho' h_\mu(w^\nu_\rho,x^{-\nu}_\rho)-\theta^\nu(u^\nu(w^\nu),x^{-\nu}_\rho) \ge -\varepsilon_{\rho'}.
	$$
    Combining the above two inequalities yields
	$$
	(\rho-\rho')h_\mu(w^\nu_\rho,x^{-\nu}_\rho)\le \varepsilon_{\rho'}+\varepsilon+\theta^\nu(w^\nu,x^{-\nu}_\rho)-\theta^\nu(u^\nu(w^\nu),x^{-\nu}_\rho)\quad\forall w^\nu\in \hat{W}^\nu_0(x^{-\nu}_\rho),
	$$
    wherein 
	letting $w^\nu=u^\nu(w^\nu)$ 
    leads to 
	$$
    \delta^\nu_\rho
     \coloneq h_\mu(w^\nu_\rho,x^{-\nu}_\rho)
    \le\frac{\varepsilon_{\rho'}+\varepsilon}{\rho-\rho'}.$$
	Furthermore, choosing arbitrary $w^\nu_{\delta^\nu_\rho}\in W^\nu$ such that
	$h_\mu(w^\nu_{\delta^\nu_\rho},x^{-\nu}_\rho)\le\delta^\nu_\rho$,
	we have
	\begin{align*}
	\theta^\nu(w^\nu_\rho,x^{-\nu}_\rho)+\rho h_\mu(w^\nu_\rho,x^{-\nu}_\rho)
	&\le \theta^\nu(w^\nu_{\delta^\nu_\rho},x^{-\nu}_\rho)+\rho
	h_\mu(w^\nu_{\delta^\nu_\rho},x^{-\nu}_\rho)+\varepsilon\\
	&\le\theta^\nu(w^\nu_{\delta^\nu_\rho},x^{-\nu}_\rho)+\rho\delta^\nu_\rho +\varepsilon,
	\end{align*}
    which together with $\delta^\nu_\rho=h_\mu(w^\nu_\rho,x^{-\nu}_\rho)$ yields 
	$$
	\theta^\nu(w^\nu_\rho,x^{-\nu}_\rho)\le\theta^\nu(w^\nu_{\delta^\nu_\rho},x^{-\nu}_\rho)+\varepsilon.
	$$
    Since 
	this inequality holds for any $w^\nu_{\delta^\nu_\rho}\in W^\nu$
	such that $h_\mu(w^\nu_{\delta^\nu_\rho},x^{-\nu}_\rho)\le\delta^\nu_\rho$, we obtain
	$$
	\theta^\nu(w^\nu_\rho,x^{-\nu}_\rho)\le\min_{w^\nu\in \hat{W}^\nu_{\delta^\nu_\rho}(x^{-\nu}_\rho)}\theta^\nu(w^\nu,x^{-\nu}_\rho)+\varepsilon
	$$
    for any $\nu$.
    Recalling
    $\delta_\rho=(\delta^1_\rho,\dots,\delta^N_\rho)$, 
    $w_\rho$ is an
        $\varepsilon$-approximate GNE for $\hatRGNEP[\delta_\rho]$.
\end{proof}

Theorem~\ref{thm:relations.NE} means that for $\{\varepsilon_k\}$ and $\{\rho_k\}$ with $\varepsilon_k\to 0$ and $\rho_k\to\infty$,
the convergent point of the sequence $\{w_k\}$ of 
$\varepsilon_k$-NE of $\rho_k$-penalized NEP is a NoLFNE.

Before proceeding to the next section, we show
Assumptions~\ref{asmp:compact.follower} and~\ref{asmp:smooth.follower} are sufficient conditions 
for $\hat{W}^\nu_0(x^{-\nu}_{\rho_k})\neq\emptyset$.
\begin{proposition}\label{prop:suff.nonempty.W}
	Suppose that Assumptions~\ref{asmp:compact.follower} and~\ref{asmp:smooth.follower} hold.
	Then, we have $\hat{W}^\nu_0(x^{-\nu})\neq\emptyset$ for any $x^{-\nu}$.
\end{proposition}
\begin{proof}
	It suffices to show that $\hat{S}(x)\neq\emptyset$.
	By the convexity of $Y$, the following generalized equation 
	is equivalently reduced to the variational inequality problem:
	\[
		0\in\nabla_{y^\omega}\gamma^\omega(x,y^\omega,y^{-\omega})+
		\mathcal{N}_{Y^\omega}(y^\omega),\quad \omega=1,\dots,M.
	\]
	Since $\nabla_{y^\omega}\gamma^\omega(x,\cdot,y^{-\omega})$ is continuous,
	the solution set of the above equation is nonempty and compact
	from~\cite[Corollary~2.2.5]{Facchinei2004}.
\end{proof}

\begin{remark}
	From Proposition~\ref{prop:suff.nonempty.W}, the assumption
	$\hat{W}^\nu_0(x^{-\nu}_{\rho_k})\neq\emptyset$ is natural.
	Another sufficient condition is that $S(x)\neq\emptyset$ for any $x\in X$,
	and this automatically holds from Proposition~\ref{prop:existence.NE}
	if $\gamma^\omega(x,\cdot,y^{-\omega})$ is convex for arbitrary fixed $(x,y^{-\omega})$ 
	under Assumption~\ref{asmp:compact.follower} for all $\omega\in\{1,\dots,M\}$.
\end{remark}


\subsection{Variational equilibrium of $\rho$-penalized NEP and its feasibility}
\label{ssec:VE}

Since $\rho$-penalized NEP is nonconvex, it may not have any NE in general.
Hence, we introduce a weaker concept of equilibrium referred to as a
\emph{variational equilibrium}. First, let us define the variational equilibrium for penalized NEP.

\begin{definition}[Variational equilibrium of $\NEP$]\label{def:VE}
	The point $w^*$ is said to be a \emph{variational equilibrium} (VE) of
	the problem $\NEP$ if the following generalized equation holds for all $\nu\in\{1,\dots,N\}$:
	\[
		0\in\nabla_{w^\nu}\Theta^\nu_\mu(w^{*,\nu},x^{*,-\nu};\rho)+
		\mathcal{N}_{W^\nu}(w^{*,\nu}).
	\]
\end{definition}

Since $W^\nu$ is compact and convex, and $\nabla_{w^\nu}\Theta^\nu_\mu$ is continuous,
the existence of the VE for $\NEP$ is 
shown in the same manner as Proposition~\ref{prop:suff.nonempty.W}.

Now, we discuss the first-order optimality condition, or necessary condition
for NoLFNE of $\hatGNEP$.
In bilevel optimization, the optimality condition is discussed over a couple of decades~\cite{Dempe2007,Dempe2014,Ye1995}.
Based on such literature, we extend their concept to MLFG.

Next, we introduce  \emph{partial calmness} for~\eqref{prob:reduced.relaxed},
which serves as a constraint qualification-like condition.
It is known that~\eqref{prob:reduced.relaxed} does not satisfy 
standard constraint qualifications for nonlinear optimization, such as the
Mangasarian--Fromovitz (MFCQ) or the
constant-rank constraint qualifications (CRCQ),
due to $h_\mu(x,y)\le 0$~\cite[Proposition~3.2]{Ye1995}.
Then, Ye and Zhu~\cite{Ye1995} proposed a weaker version of
calmness~\cite[Definition~6.4.1]{Clarke1990} for level-reduced bilevel optimization.

To introduce the partial calmness condition of our context,
we let $p^\nu\in \Re$ for each $\nu\in \{1,\ldots,N\}$ and consider the following perturbed problem for~\eqref{prob:reduced.relaxed}:
\begin{align}\label{prob:reduced.perturbed}
\min_{w^\nu}\quad \theta^\nu(w^\nu,x^{-\nu})  \qquad
\text{s.t.} \quad w^\nu\in W^\nu,\ h_\mu(w^\nu,x^{-\nu})+p^\nu=0.
\end{align}
The difference between~\eqref{prob:reduced.relaxed} and~\eqref{prob:reduced.perturbed} lies in the perturbation of the constraint $h_\mu(w^\nu,x^{-\nu})\le 0$. 
The partial calmness condition is defined as follows.

\begin{definition}[Partial calmness]\label{def:partially.calm}
	We say that~\eqref{prob:reduced.relaxed} is \emph{partially calm} at $\bar{w}^\nu\in\hat{W}^\nu_0(\bar{x}^{-\nu})$ for a fixed $\bar{x}^{-\nu}$ if there exist $\rho^\nu>0$
	and a neighborhood $U^\nu\in\Re^{n_\nu+m}\times\Re$ of $(\bar{w}^\nu,0)\in\Re^{n_\nu+m}\times\Re$ such that for all
	$(w^\nu,p^\nu)\in U^\nu$ that is feasible to~\eqref{prob:reduced.perturbed}, the following inequality holds:
	\begin{align}\label{ieq:partially.calm}
		\theta^\nu(w^\nu,x^{-\nu})-\theta^\nu(\bar{w}^\nu,x^{-\nu})+\rho^\nu|p^\nu| \ge 0.
	\end{align}
	If the above statement holds for all $\nu\in\{1,\dots,N\}$, we say that constrained GNEP\\
    ($\hatGNEP$) is partially calm at $\bar{w}\in\hat{W}_0(\bar{x})$.
\end{definition}

Recalling the first assertion of 
Proposition~\ref{prop:NIfunc}, i.e., $h_\mu\ge 0$, note that inequality~\eqref{ieq:partially.calm} can also be written as follows if 
$w^\nu$ is feasible to~\eqref{prob:reduced.perturbed}:
$$
\theta^\nu(w^\nu,x^{-\nu})-\theta^\nu(\bar{w}^\nu,x^{-\nu})+\rho^\nu h_\mu(w^\nu,x^{-\nu})\ge 0.
$$
This inequality means that the variation in the leader $\nu$'s objective
from $\bar{w}^\nu$ to ${w}^\nu$ being feasible to \eqref{prob:reduced.perturbed}, $\theta^\nu({w}^\nu,x^{-\nu})-\theta^\nu(\bar{w}^\nu,x^{-\nu})$, 
is bounded below by $\rho^\nu$ times some perturbation for $-h_\mu$.
It can also be seen that the partial calmness condition, in practice, requires 
$\hatGNEP$ to admit a Lagrange multiplier associated with the inequality constraint $h_\mu(w^\nu,x^{-\nu})\le 0$ for each leader's reduced problem~\eqref{prob:reduced.relaxed}.

The following lemma implies that for a given point $x^*\in X$, 
if $h_\mu(\cdot,x^{*,-\nu})$ satisfies HEB for all $\nu$ and the objective function
is locally Lipschitz continuous around a feasible point of $h_\mu(w^\nu,x^{*,-\nu})\le 0$,~\eqref{prob:reduced.relaxed} is partially calm at $w^{*,\nu}\in \hat{W}^\nu_0(x^{*,-\nu})$.

\begin{lemma}\label{lem:HEB.partially.calm}
	Let $x^*\in X$ be given.
    Suppose that $\hat{W}^\nu_0(x^{*,-\nu})\neq\emptyset$ for all $\nu\in\{1,\dots,N\}$, and Assumption~\ref{asmp} holds.
	For $\varepsilon<1$, consider $p^\nu\in\Re$ with $|p^\nu|\le\varepsilon$ and any $w^\nu\in W^{\nu}$ such that
\begin{align*}
h_\mu(w^\nu,x^{*,-\nu})+p^\nu=0
\mbox{ and }
\|{w}^{*,\nu} - w^\nu\|\le \varepsilon
\end{align*}
for some 
	\begin{align*}
	w^{*,\nu}\in\argmin_{\bar{w}^\nu\in W^\nu}\{ \|\bar{w}^\nu - w^\nu\| \mid h_\mu(\bar{w}^\nu,x^{*,-\nu})\le 0\}.
\end{align*}
	If $h_\mu(\cdot,x^{*,-\nu})$ satisfies the $(\xi,\eta)$-HEB\footnote{Recalling Proposition~\ref{prop:NIfunc}, note that from the assumption $\hat{W}^\nu_0(x^{*,-\nu})\neq\emptyset$, we have $\min_{w^\nu} h_\mu(w^\nu,x^{*,-\nu}) = 0$.} that $h_\mu(w^\nu,x^{*,-\nu})\ge\xi\|w^\nu-w^{*,\nu}\|^\eta$
	with $\eta\le 1$,
	then~\eqref{prob:reduced.relaxed} is partially calm at $w^{*,\nu}$.
\end{lemma}
\begin{proof}
	By definition, for $w^\nu$ such that $\|w^\nu-w^{*,\nu}\|\le \varepsilon$ with $\varepsilon<1$, we have
	$$
	\varepsilon \ge |p^\nu|=h_\mu(w^\nu,x^{*,-\nu})\ge \xi\|w^\nu-w^{*,\nu}\|^\eta\ge\xi\|w^\nu-w^{*,\nu}\|,
	$$
	where the last inequality holds by the assumptions
	that $\|w^\nu-w^{*,\nu}\| <1$, $\xi > 0$ and $\eta\le 1$.
	Then, it follows that
	$$
	\theta^\nu(w^\nu,x^{*,-\nu})-\theta^\nu(w^{*,\nu},x^{*,-\nu})
	\ge -\bar{L}\|w^\nu-w^{*,\nu}\|
	\ge -\frac{\bar{L}}{\xi}|p^\nu|,
	$$
	where the first inequality holds from Lemma~\ref{lem:smooth.lipschitz}.
	Therefore, the partial calmness condition holds with constant $\bar{L}/\xi$.
\end{proof}

The assumption of the HEB with $\eta\le 1$ in Lemma~\ref{lem:HEB.partially.calm} requires that
$h_\mu(\cdot,x^{*,-\nu})$ converges to zero no faster than $\xi\|w^\nu-w^{*,\nu}\|^\eta$ with $\eta\le 1$.

As stated below, if $\hatGNEP$ is partially calm at some NoLFNE (see Definition~\ref{def:normal.LFNE}),
then the first-order optimality conditions hold for all~$\nu$.

\begin{theorem}\label{thm:optimality}
	Let $(x^*,y^*)\in X\times \hat{S}(x^*)$ be a NoLFNE and $w^*\in X\times \hat{S}(x^*)^N$.
	Suppose that Assumption~\ref{asmp} holds and that
    $\hatGNEP$ is partially calm at $w^*$.
	Then, there exists $\rho^*>0$ such that for all $\nu\in\{1,\dots,N\}$, 
	\begin{align}\label{ge:optimality.condition}
		0\in\nabla_{w^\nu}\theta^\nu(w^{*,\nu},x^{*,-\nu})+
		\rho\nabla_{w^\nu}h_\mu(w^{*,\nu},x^{*,-\nu})+
		\mathcal{N}_{W^\nu}(w^{*,\nu})
	\end{align}
    holds for any $\rho> \rho^*$.
\end{theorem}
\begin{proof}
	Since~\eqref{prob:reduced.relaxed} is partially calm at $w^{*,\nu}$,
	\cite[Proposition~2.2]{Ye1997} ensures that there exists $\rho^*_\nu>0$ such that 
	for all $\rho^\nu>\rho^*_\nu>0$, $w^{*,\nu}$ solves
	$$
	\min_{w^\nu\in W^\nu}\quad\theta^\nu(w^\nu,x^{*,-\nu})+\rho^\nu h_\mu(w^\nu,x^{*,-\nu}).
	$$
	Hence, the first-order optimality condition for the above problem under the convexity of $W^\nu$ is obtained as follows:
	$$
		0\in\nabla_{w^\nu}\theta^\nu(w^{*,\nu},x^{*,-\nu})+
		\rho^\nu\nabla_{w^\nu}h_\mu(w^{*,\nu},x^{*,-\nu})+
		\mathcal{N}_{W^\nu}(w^{*,\nu}).
	$$
	Therefore, the generalized equation holds for all $\nu$.
    Finally, defining $\rho^*>\max_{\nu=1,\dots,N} \rho^*_\nu$ leads to~\eqref{ge:optimality.condition}.
\end{proof}

\begin{corollary}\label{cor:HEB.VE}
	Let $(x^*,y^*)\in X\times \hat{S}(x^*)$ be a NoLFNE and $w^*\in X\times \hat{S}(x^*)^N$.
	Suppose that Assumption~\ref{asmp} and the assumptions of Lem\-ma~\ref{lem:HEB.partially.calm} hold. 
	Then, there exists $\rho^*>0$ such that~\eqref{ge:optimality.condition} holds
	for all $\nu\in\{1,\dots,N\}$.
\end{corollary}



In the rest of this section, we establish the relationship between
VE of $\NEP$ and  $\hatRGNEP$ and also show that the constraint $h_\mu(w^\nu,x^{-\nu})\le 0$
is approximately satisfied at a VE of $\NEP$.

\begin{definition}
    \label{def:VE.RGNEP}
    A point $w_\delta$ is called a \emph{VE} of $\hatRGNEP$ if there exists $\zeta^\nu_\delta\ge 0$ such that 
    $w^\nu_\delta\in\hat{W}^\nu_{\delta^\nu}(x^{-\nu}_\delta)$ satisfies
    \begin{align}\label{eq:VE.full}
    \begin{gathered}
        0\in\nabla_{w^\nu}\theta^\nu(w^\nu_\delta,x^{-\nu}_\delta)+\zeta^\nu_\delta \nabla_{w^\nu}h_\mu(w^\nu_\delta,x^{-\nu}_\delta)+\mathcal{N}_{W^\nu}(w^\nu_\delta),\\
        0\le \zeta^\nu_\delta \perp h_\mu(w^\nu_\delta,x^{-\nu}_\delta)-\delta^\nu\le 0,
    \end{gathered}
    \end{align}
    for all $\nu\in\{1,\dots,N\}$.
\end{definition}

The following theorem shows that the VE of $\NEP$ is also that of \\ $\hatRGNEP$.
\begin{theorem}\label{thm:VE.rel}
    Let $w_\rho\in W$ be a VE of $\NEP$, i.e.,
    \begin{align*}
        0\in \nabla_{w^\nu} \theta^\nu(w^\nu_\rho,x^{-\nu}_\rho)+
        \rho\nabla_{w^\nu} h_\mu(w^\nu_\rho,x^{-\nu}_\rho) + \mathcal{N}_{W^\nu}(w^\nu_\rho)
    \end{align*}
    for all $\nu\in\{1,\ldots,N\}$.
    Then, $w_\rho$ is a VE of $\hatRGNEP$ with $\delta^\nu=h_\mu(w^\nu_\rho,x^{-\nu}_\rho)$ for all $\nu$.
\end{theorem}
\begin{proof}
The assertion readily follows 
by setting $w_\delta = w_\rho$ and $\zeta^\nu_\delta=\rho$ with $\delta^\nu=h_\mu(w^\nu_\rho,x^{-\nu}_\rho)$ for $\nu=1,\dots,N$ in Definition~\ref{def:VE.RGNEP}.
\end{proof}

Lastly, we conclude this section with the proposition concerning the bounds on the feasibility violation of $h_\mu(w^\nu,x^{-\nu})\le 0$ around an approximate VE of the $\rho$-penalized NEP.
We make the following assumption, which imposes 
a Kurdyka--{\L}ojasiewicz-like property on $h_{\mu}$.
\begin{assumption}\label{asmp:jointKL}
    For all $\nu\in\{1,\dots,N\}$ the following conditions hold:
    Let $w^*\in W$ be such that $h_\mu(w^{*,\nu},x^{*,-\nu})=0$.
    Suppose that there exist $\alpha>0$, $\beta\in[0,1)$, and a neighborhood $\mathcal{U}^\nu$
    of $(w^{*,\nu},x^{*,-\nu})\in W^\nu\times X^{-\nu}$, where $X^{-\nu}\coloneq\prod_{\nu'\neq \nu} X^{\nu'}$,
    such that for all $(w^\nu,x^{-\nu})\in \mathcal{U}^\nu\cap (W^\nu\times X^{-\nu})$,
    \begin{align}\label{ieq:KLinequality}
        \alpha h_\mu(w^\nu,x^{-\nu})^\beta \le \dist(0,\nabla_{w^\nu}h_\mu(w^\nu,x^{-\nu})+\mathcal{N}_{W^\nu}(w^\nu)).
    \end{align}
\end{assumption}


\begin{proposition}\label{prop:feasibility.VE}
    For $\varepsilon>0$, let $w_\rho\in W$ be such that
	\begin{align}\label{ieq:approx.VE}
		\dist(0,\nabla_{w^\nu}\theta^\nu(w^\nu_\rho,x^{-\nu}_\rho)
		+\rho\nabla_{w^\nu}h_\mu(w^\nu_\rho,x^{-\nu}_\rho)
		+\mathcal{N}_{W^\nu}(w^\nu_\rho)) \le \varepsilon
	\end{align}
	for all $\nu\in\{1,\dots,N\}$.
	Suppose that Assumptions~\ref{asmp} and~\ref{asmp:jointKL} hold.
	Then,
	\begin{equation}
		h_\mu(w^\nu_\rho,x^{-\nu}_\rho)\le\left(\frac{\eps+M_\nu}{\alpha \rho}\right)^{\frac1\beta}\label{eq:prop42}
\end{equation}
	for all $\nu$, where $M_\nu \coloneq \max \|\nabla_{w^\nu}\theta^\nu(w^\nu,x^{-\nu})\|$.
\end{proposition}
\begin{proof}
	Let 
	\[
		\xi^\nu_\rho\in
		\nabla_{w^\nu}\theta^\nu(w^\nu_\rho,x^{-\nu}_\rho)+
		\rho\nabla_{w^\nu}h_\mu(w^\nu_\rho,x^{-\nu}_\rho)+\mathcal{N}_{W^\nu}(w^\nu_\rho)
		\ \text{with}\ 
		\|\xi^\nu_\rho\|\le\eps.
	\]
	Then, we have
	\[
		\begin{aligned}
		\frac{\xi^\nu_\rho-\nabla_{w^\nu}\theta^\nu(w^\nu_\rho,x^{-\nu}_\rho)}{\rho}
		&\in\nabla_{w^\nu}h_\mu(w^\nu_\rho,x^{-\nu}_\rho)+
		\frac1\rho\mathcal{N}_{W^\nu}(w^\nu_\rho)\\
		&=\nabla_{w^\nu}h_\mu(w^\nu_\rho,x^{-\nu}_\rho)+\mathcal{N}_{W^\nu}(w^\nu_\rho),
		\end{aligned}
	\]
	where the last equality holds because $\mathcal{N}_{W^\nu}(w^\nu_\rho)$ is a cone.
	Hence, we have
	\[
		\begin{aligned}
		&\dist(0,\nabla_{w^\nu} h_\mu(w^\nu_\rho,x^{-\nu}_\rho)+\mathcal{N}_{W^\nu}(w^\nu_\rho))
		\le\left\|
			\frac{\xi^\nu_\rho-\nabla_{w^\nu}\theta^\nu(w^\nu_\rho,x^{-\nu}_\rho)}{\rho}
		\right\|\\
		\le&
		\frac{\|\xi^\nu_\rho\|+\|\nabla_{w^\nu}\theta^\nu(w^\nu_\rho,x^{-\nu}_\rho)\|}{\rho}
		\le\frac{\eps+M_\nu}{\rho},
		\end{aligned}
	\]
    from which together with Assumption~\ref{asmp:jointKL} the desired inequality\,\eqref{eq:prop42} follows.
\end{proof}

Finally, we summarize the relationships between the MLFG and all the NEPs in Figure~\ref{fig:rel} and Table~\ref{tab:rel}.
\begin{figure}[h]
    \centering
    \begin{tikzpicture}[
  scale = 0.75,
  box/.style={
    draw,
    rounded corners=5pt,
    minimum width=3.8cm,
    minimum height=0.6cm,
    align=center,
    thick,
    font=\normalsize
  }
]

\node[box] (mlfg) at (0,  0)   {MLFG (original)};
\node[box] (gnep) at (0, -1.5) {$\GNEP$};
\node[box] (hatgnep) at (0, -3.0) {$\hatGNEP$};
\node[box] (nep) at (-4.5, -5.0) {$\NEP$};
\node[box] (rgnep) at (4.5, -5.0) {$\hatRGNEP$};

\draw[Implies-Implies, double equal sign distance, thick]
  (mlfg) -- 
  node[midway, left, align=right, font=\small]
  {equivalent}
  (gnep);

\draw[-Implies, double equal sign distance, thick]
  ([xshift=-10pt]gnep.south) --
  node[midway, left, align=right, font=\small]
  {necessary}
  ([xshift=-10pt]hatgnep.north);

\draw[-Implies, double equal sign distance, thick]
  ([xshift=10pt]hatgnep.north) --
  node[midway, right, align=left, font=\small]
  {sufficient(under convextiy)}
  ([xshift=10pt]gnep.south);

\draw[-{Stealth}, thick]
  (hatgnep) --
  node[midway, left,  align=right, font=\small, xshift=-1cm]
    {penalized\\$h_\mu(w^\nu,x^{-\nu}) \le 0$}
  (nep);

\draw[-{Stealth}, thick]
  (hatgnep) --
  node[midway, right, align=left,  font=\small, xshift=1cm]
    {relaxed\\$h_\mu(w^\nu,x^{-\nu}) \le 0$}
  (rgnep);

\draw[-{Stealth}, thick]
  (nep) -- 
  node[midway, below, font=\small] {Table~\ref{tab:rel}}
  (rgnep);
  
\end{tikzpicture}
    \caption{Relationships between various noncooperative games}
    \label{fig:rel}
\end{figure}
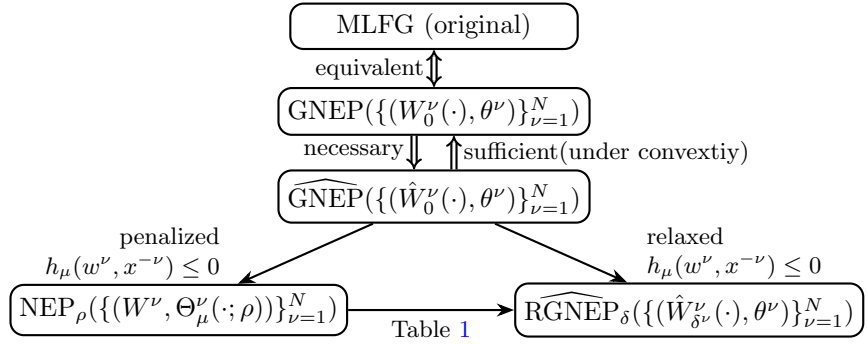
\begin{table}[h]
    \centering
    \caption{Relationships between $\NEP$~and~$\hatRGNEP$}
    \label{tab:rel}
    \begin{tabular}{ccc}
         $\NEP$  & (rel.) & $\hatRGNEP$ \\\hline
         $\varepsilon$-NE at $w_\rho$ & \begin{tabular}[c]{@{}c@{}} $\implies$ \\ Assump.~\ref{asmp} \&~\ref{asmp:subanalyticity} \end{tabular} & \begin{tabular}[c]{@{}c@{}}$\varepsilon$-GNE with $\delta_\rho$ at $w_\rho$\\ ($\delta^\nu_\rho=h_\mu(w^\nu_\rho,x^{-\nu}_\rho)$)\end{tabular} \\
         VE at $w_\rho$ & $\implies$ & \begin{tabular}[c]{@{}c@{}} VE at $w_\rho$ \\ ($\delta^\nu=h_\mu(w^\nu_\rho,x^{-\nu}_\rho)$)\end{tabular} \\
         $\varepsilon$-approx. VE & \begin{tabular}[c]{@{}c@{}} $\implies$ \\ Assump.~\ref{asmp} \&~\ref{asmp:jointKL} \end{tabular} & $h_\mu(w^\nu_\rho,x^{-\nu}_\rho)\le\left(\frac{\varepsilon+M_\nu}{\alpha\rho}\right)^{1/\beta}$
    \end{tabular}
\end{table}

\section{Conclusions}\label{sec:conclusion}

This paper proposed a new reformulation for the MLFG by exploiting
the regularized Nikaido--Isoda function.
Unlike traditional reformulations such as EPEC and response-based 
approaches, our reformulation does not use any derivative information
regarding each follower's problem.
We analyzed the relation between the (No)LFNE and the approximated
NE for the resulting NEP.
Although we assumed that the followers' constraint sets are
independent of~$x$ throughout this paper,
readers may be interested in the more general case of $Y^\omega(x)$.
However, the analysis may become much more difficult since it may require
the set-valued analysis as well.
We also leave it for future work.


\vspace{10pt}

\appendix  

\section*{Appendix: Implicit differentiation of the followers' response}

Using Fischer--Burmeister functions~\cite{Facchinei2004}, the complementarity conditions of KKT system~\eqref{KKT.system} are equivalently 
reformulated as a nonsmooth equation as follows:
\[
0=H_0(x,y,\lambda)  \coloneq 
\begin{bmatrix}
    \nabla_{y^\omega}\gamma^\omega(x,y^\omega,y^{-\omega})+\nabla_{y^\omega} g^\omega(y^\omega)\lambda^\omega \\
    [\phi(\lambda^\omega_i,-g^\omega_i(y^\omega))]_{i=1}^{p_\omega}
\end{bmatrix}_{\omega=1}^M \in\Re^{m+p},
\]
where $p \coloneq p_1+\dots+p_M$ and $\phi\colon\Re^2\to\Re$ satisfies  $\phi(a,b)=0$ if and only if $0\le a \perp b\ge 0$.
Since $\phi(a,b)$ is not differentiable at some points in general, a smoothing approximated Fischer--Burmeister function $\phi_\varepsilon$,
where $\varepsilon>0$ denotes an approximation parameter, is adopted such that $\lim_{\varepsilon\searrow 0}\phi_\varepsilon(a,b)=\phi(a,b)$ for any $a,b\in\Re^2$.
Then, by replacing $\phi$ with $\phi_\varepsilon$, the function $H_0$ is approximated by the smoothed system of equations:
$H_\varepsilon(x,y,\lambda)=0$.

Under the uniqueness assumption of $y$ and $\lambda$ for $H_\varepsilon(x,y,\lambda)=0$, the solution to the equation
with respect to $y$ and $\lambda$ is written as $y_\varepsilon(x)$ and $\lambda_\varepsilon(x)$.
By using the implicit function theorem~\cite[Proposition~4.5]{Hori2024},
the Jacobian $[\nabla y_\varepsilon(x), \nabla \lambda_\varepsilon(x)]$ is computed as
\[
[\nabla y_\varepsilon(x), \nabla \lambda_\varepsilon(x)]
=-\nabla_x H_\varepsilon(x,y_\varepsilon(x),\lambda_\varepsilon(x))
\begin{bmatrix}
    \nabla_y H_\varepsilon(x,y_\varepsilon(x),\lambda_\varepsilon(x)) \\
    \nabla_\lambda H_\varepsilon(x,y_\varepsilon(x),\lambda_\varepsilon(x))
\end{bmatrix}^{-1}.
\]
This means that the gradient computation of the (implicit) response requires Hessian matrices
$\nabla^2_y \gamma^\omega$ and $\nabla^2_{y^\omega}g^\omega$ and inverse matrix as shown above.

\bibliographystyle{abbrv}
\bibliography{reference.bib}

@article{Allevi2018, 
  year     = {2018}, 
  title    = {On an equilibrium problem with complementarity constraints formulation of pay-as-clear electricity market with demand elasticity}, 
  author   = {Allevi, Elisabetta and Aussel, Didier and Riccardi, Rossana}, 
  journal  = {Journal of Global Optimization}, 
  issn     = {0925-5001}, 
  doi      = {10.1007/s10898-017-0595-9}, 
  pages    = {329--346}, 
  number   = {2}, 
  volume   = {70}
}

@article{Ashraf2023, 
  year     = {2023}, 
  title    = {Optimal resource allocation strategies of competing new online delivery platforms using the Bass diffusion model}, 
  author   = {Ashraf, Saad and Bardhan, Amit Kumar}, 
  journal  = {Annals of Operations Research}, 
  issn     = {0254-5330}, 
  doi      = {10.1007/s10479-023-05410-6}, 
  pages    = {1--20}
}

@book{Aubin1979, 
  year = {1979},
  title     = {Mathematical Methods of Game and Economic Theory}, 
  author    = {Aubin, Jean Pierre}, 
  isbn      = {9780486462653}, 
  publisher = {North-Holland, Amsterdam}
}

@inproceedings{Aussel2020,
  author = {Aussel, D and Svensson, A},
  booktitle = {Bilevel Optimization---Advances and Next Challenges},
  editor = {Dempe, S and Zemkoho, A},
  title = {A short state of the art on multi-leader--follower games},
  publisher = {Springer International Publishing},
  address = {Cham},
  year = {2020},
}

@article{Bierstone1988, 
  year    = {1988}, 
  title   = {Semianalytic and subanalytic sets}, 
  author  = {Bierstone, Edward and Milman, Pierre D.}, 
  journal = {Publications Mathématiques de l'Institut des Hautes Études Scientifiques}, 
  issn    = {0073-8301}, 
  doi     = {10.1007/bf02699126}, 
  pages   = {5--42}, 
  number  = {1}, 
  volume  = {67}
}

@article{Bolte2007, 
  year     = {2007}, 
  title    = {The {{\L}}ojasiewicz Inequality for Nonsmooth Subanalytic Functions with Applications to Subgradient Dynamical Systems}, 
  author   = {Bolte, Jerome and Daniilidis, Aris and Lewis, Adrian}, 
  journal  = {{SIAM} Journal on Optimization}, 
  issn     = {1052-6234}, 
  doi      = {10.1137/050644641}, 
  pages    = {1205--1223}, 
  number   = {4}, 
  volume   = {17}
}

@book{Bonnans2000, 
  year     = {2000}, 
  title    = {Perturbation Analysis of Optimization Problems}, 
  author   = {Bonnans, J. Frédéric and Shapiro, Alexander}, 
  publisher= {Springer New York, NY},
  isbn     = {9781461271291}, 
  doi      = {10.1007/978-1-4612-1394-9}
}

@InProceedings{Chen202567l,
  title = 	 {Efficient First-Order Optimization on the Pareto Set for Multi-Objective Learning under Preference Guidance},
  author =       {Chen, Lisha and Xiao, Quan and Fukuda, Ellen Hidemi and Chen, Xinyi and Yuan, Kun and Chen, Tianyi},
  booktitle = 	 {Proceedings of the 42nd International Conference on Machine Learning},
  pages = 	 {9443--9486},
  year = 	 {2025},
  editor = 	 {Singh, Aarti and Fazel, Maryam and Hsu, Daniel and Lacoste-Julien, Simon and Berkenkamp, Felix and Maharaj, Tegan and Wagstaff, Kiri and Zhu, Jerry},
  volume = 	 {267},
  series = 	 {Proceedings of Machine Learning Research},
  month = 	 {13--19 Jul},
  publisher =    {PMLR},
  url = 	 {https://proceedings.mlr.press/v267/chen25bw.html}
}

@book{Clarke1990, 
  year      = {1990}, 
  title     = {Optimization and Nonsmooth Analysis}, 
  author    = {Clarke, Frank H}, 
  isbn      = {9780898712568}, 
  url       = {http://epubs.siam.org/doi/book/10.1137/1.9781611971309}, 
  publisher = {Society for Industrial and Applied Mathematics}, 
  doi       = {10.1137/1.9781611971309}
}

@article{Dempe2007, 
  year     = {2007}, 
  title    = {New necessary optimality conditions in optimistic bilevel programming}, 
  author   = {Dempe, S. and Dutta, J. and Mordukhovich, B. S.}, 
  journal  = {Optimization}, 
  issn     = {0233-1934}, 
  doi      = {10.1080/02331930701617551}, 
  pages    = {577--604}, 
  number   = {5-6}, 
  volume   = {56}
}

@article{Dempe2014, 
  year     = {2014}, 
  title    = {Necessary optimality conditions in pessimistic bilevel programming}, 
  author   = {Dempe, S. and Mordukhovich, B.S. and Zemkoho, A.B.}, 
  journal  = {Optimization}, 
  issn     = {0233-1934}, 
  doi      = {10.1080/02331934.2012.696641}, 
  pages    = {505--533}, 
  number   = {4}, 
  volume   = {63}
}

@article{Dries1996, 
  year    = {1996}, 
  title   = {Geometric categories and O-minimal structures}, 
  author  = {Dries, L. and Miller, C.}, 
  journal = {Duke Mathematical Journal}, 
  pages   = {497--540}, 
  volume  = {84}
}

@inproceedings{Dutta2006,
  address   = {Boston, MA},
  author    = {Dutta, Joydeep
               and Dempe, Stephan},
  booktitle = {Optimization with Multivalued Mappings: Theory, Applications, and Algorithms},
  doi       = {10.1007/0-387-34221-4_3},
  editor    = {Dempe, Stephan
               and Kalashnikov, Vyacheslav},
  isbn      = {978-0-387-34221-4},
  pages     = {51--71},
  publisher = {Springer US},
  title     = {Bilevel programming with convex lower level problems},
  url       = {https://doi.org/10.1007/0-387-34221-4_3},
  year      = {2006}
}

@article{Facchinei1999, 
  year     = {1999}, 
  title    = {A smoothing method for mathematical programs with equilibrium constraints}, 
  author   = {Facchinei, Francisco and Jiang, Houyuan and Qi, Liqun}, 
  journal  = {Mathematical Programming}, 
  issn     = {0025-5610}, 
  doi      = {10.1007/s10107990015a}, 
  pages    = {107--134}, 
  number   = {1}, 
  volume   = {85}
}

@book{Facchinei2004, 
  year      = {2004}, 
  title     = {Finite-Dimensional Variational Inequalities and Complementarity Problems}, 
  author    = {Facchinei, Francisco and Pang, Jong-Shi}, 
  isbn      = {9780387955810}, 
  url       = {http://link.springer.com/10.1007/b97544}, 
  series    = {Springer Series in Operations Research and Financial Engineering}, 
  publisher = {Springer New York}, 
  address   = {New York, {NY}}, 
  doi       = {10.1007/b97544}
}

@article{Facchinei2010, 
  year     = {2010}, 
  title    = {Generalized {N}ash Equilibrium Problems}, 
  author   = {Facchinei, Francisco and Kanzow, Christian}, 
  journal  = {Annals of Operations Research}, 
  issn     = {0254-5330}, 
  doi      = {10.1007/s10479-009-0653-x}, 
  pages    = {177--211}, 
  number   = {1}, 
  volume   = {175}
}

@INPROCEEDINGS{Franci2025,
  author={Franci, Barbara and Fabiani, Filippo and Schmidt, Martin and Staudigl, Mathias},
  booktitle={2025 European Control Conference (ECC)}, 
  title={A {G}auss–{S}eidel method for solving multi-leader-multi-follower games}, 
  year={2025},
  volume={},
  number={},
  pages={2775-2780},
  doi={10.23919/ECC65951.2025.11187084}
}

@article{Gurkan2009, 
  year     = {2009}, 
  title    = {Approximations of {N}ash equilibria}, 
  author   = {Gürkan, Gül and Pang, Jong-Shi}, 
  journal  = {Mathematical Programming}, 
  issn     = {0025-5610}, 
  doi      = {10.1007/s10107-007-0156-y}, 
  pages    = {223--253}, 
  number   = {1-2}, 
  volume   = {117}
}

@article{Herty2022, 
  year     = {2022}, 
  title    = {Solving quadratic multi-leader-follower games by smoothing the follower's best response}, 
  author   = {Herty, Michael and Steffensen, Sonja and Thünen, Anna}, 
  journal  = {Optimization Methods and Software}, 
  issn     = {1055-6788}, 
  doi      = {10.1080/10556788.2020.1828412}, 
  pages    = {772--799}, 
  number   = {2}, 
  volume   = {37}
}

@article{Heusinger2009, 
  year     = {2009}, 
  title    = {Optimization reformulations of the generalized {N}ash equilibrium problem using {N}ikaido-{I}soda-type functions}, 
  author   = {von Heusinger, Anna and Kanzow, Christian}, 
  journal  = {Computational Optimization and Applications}, 
  issn     = {0926-6003}, 
  doi      = {10.1007/s10589-007-9145-6}, 
  pages    = {353--377}, 
  number   = {3}, 
  volume   = {43}
}

@article{Hori2019, 
  year     = {2019}, 
  title    = {Gauss–{S}eidel Method for Multi-leader–follower Games}, 
  author   = {Hori, Atsushi and Fukushima, Masao}, 
  journal  = {Journal of Optimization Theory and Applications}, 
  issn     = {0022-3239}, 
  doi      = {10.1007/s10957-018-1391-5}, 
  pages    = {651--670}, 
  number   = {2}, 
  volume   = {180}
}

@article{Hori2024, 
  year     = {2024}, 
  title    = {A Method for Multi-Leader–Multi-Follower Games by Smoothing the Followers’ Response Function}, 
  author   = {Hori, Atsushi and Tsuyuguchi, Daisuke and Fukuda, Ellen H.}, 
  journal  = {Journal of Optimization Theory and Applications}, 
  issn     = {0022-3239}, 
  doi      = {10.1007/s10957-024-02506-2}, 
  pages    = {305--335}, 
  number   = {1}, 
  volume   = {203}
}

@article{Hu2011, 
  year     = {2011}, 
  title    = {Variational Inequality Formulation of a Class of Multi-Leader-Follower Games}, 
  author   = {Hu, Ming and Fukushima, Masao}, 
  journal  = {Journal of Optimization Theory and Applications}, 
  issn     = {0022-3239}, 
  doi      = {10.1007/s10957-011-9901-8}, 
  pages    = {455}, 
  number   = {3}, 
  volume   = {151}
}

@article{Hu2015, 
  year     = {2015}, 
  title    = {Multi-Leader-Follower Games: Models, Methods and Applications}, 
  author   = {Hu, Ming and Fukushima, Masao}, 
  journal  = {Journal of the Operations Research Society of Japan}, 
  issn     = {0453-4514}, 
  doi      = {10.15807/jorsj.58.1}, 
  pages    = {1--23}, 
  number   = {1}, 
  volume   = {58}
}

@article{Huang2006, 
  year     = {2006}, 
  title    = {A Sequential Smooth Penalization Approach to Mathematical Programs with Complementarity Constraints}, 
  author   = {Huang, X. X. and Yang, X. Q. and Zhu, D. L.}, 
  journal  = {Numerical Functional Analysis and Optimization}, 
  issn     = {0163-0563}, 
  doi      = {10.1080/01630560500538797}, 
  pages    = {71--98}, 
  number   = {1}, 
  volume   = {27}
}

@article{Kim2018, 
  year     = {2018}, 
  title    = {A Novel Bitcoin Mining Scheme Based on the Multi-Leader Multi-Follower {S}tackelberg Game Model}, 
  author   = {Kim, Sungwook}, 
  journal  = {{IEEE} Access}, 
  issn     = {2169-3536}, 
  doi      = {10.1109/access.2018.2867631}, 
  pages    = {48902--48912}, 
  volume   = {6}
}

@article{Kosiba2025, 
  year     = {2025}, 
  title    = {The generalized Łojasiewicz inequality for definable and subanalytic multifunctions}, 
  author   = {Kosiba, Michał}, 
  journal  = {Journal of Mathematical Analysis and Applications}, 
  issn     = {0022-247X}, 
  doi      = {10.1016/j.jmaa.2024.128977}, 
  pages    = {128977}, 
  number   = {2}, 
  volume   = {543}
}

@article{kumar2022, 
  year     = {2022}, 
  title    = {A bilevel game model for ascertaining competitive target prices for a buyer in negotiation with multiple suppliers}, 
  author   = {Kumar, Akhilesh and Gupta, Anjana and Mehra, Aparna}, 
  journal  = {{RAIRO} - Operations Research}, 
  issn     = {0399-0559}, 
  doi      = {10.1051/ro/2021185}, 
  pages    = {293--330}, 
  number   = {1}, 
  volume   = {56}
}

@article{Lei2023, 
  year     = {2023}, 
  title    = {A {N}ash–{S}tackelberg game approach to analyze strategic bidding for multiple {DER} aggregators in electricity markets}, 
  author   = {Lei, Zhenxing and Liu, Mingbo and Shen, Zhijun and Lu, Junqi and Lu, Zhilin}, 
  journal  = {Sustainable Energy, Grids and Networks}, 
  issn     = {2352-4677}, 
  doi      = {10.1016/j.segan.2023.101111}, 
  pages    = {101111}, 
  volume   = {35}
}

@article{Leyffer2010, 
  year   = {2010}, 
  title  = {Solving multi-leader–common-follower games}, 
  journal= {Optimization Methods and Software},
  author = {Leyffer, Sven and Munson, Todd}, 
  pages  = {601--623}, 
  volume = {25}
}

@article{Liu2024, 
  year     = {2024}, 
  title    = {Moreau Envelope for Nonconvex Bi-Level Optimization: A Single-loop and {H}essian-free Solution Strategy}, 
  author   = {Liu, Risheng and Liu, Zhu and Yao, Wei and Zeng, Shangzhi and Zhang, Jin}, 
  journal  = {{arXiv}}, 
  doi      = {10.48550/arxiv.2405.09927}, 
  eprint   = {2405.09927}
}

@article{Lu2024, 
  year    = {2024}, 
  title   = {First-Order Penalty Methods for Bilevel Optimization}, 
  author  = {Lu, Zhaosong and Mei, Sanyou}, 
  journal = {{SIAM} Journal on Optimization}, 
  issn    = {1052-6234}, 
  doi     = {10.1137/23m1566753}, 
  pages   = {1937--1969}, 
  number  = {2}, 
  volume  = {34}
}

@INPROCEEDINGS{Lyu2022,
  author={Lyu, Ting and Xu, Haitao and Han, Zhu},
  booktitle={ICC 2022 - IEEE International Conference on Communications}, 
  title={Multi-leader Multi-follower Stackelberg Game based Resource Allocation in Multi-access Edge Computing}, 
  year={2022},
  volume={},
  number={},
  pages={4306-4311},
  doi={10.1109/ICC45855.2022.9838425}
}

@article{Nash1950, 
  year    = {1950}, 
  title   = {Equilibrium points in n-person games}, 
  author  = {Nash, John F.}, 
  journal = {Proceedings of the National Academy of Sciences}, 
  issn    = {0027-8424}, 
  doi     = {10.1073/pnas.36.1.48}, 
  pmid    = {16588946}, 
  pmcid   = {{PMC}1063129}, 
  pages   = {48--49}, 
  number  = {1}, 
  volume  = {36}
}

@article{Pang2005n, 
  year     = {2005}, 
  title    = {Quasi-variational inequalities, generalized {N}ash equilibria, and multi-leader-follower games}, 
  author   = {Pang, Jong-Shi and Fukushima, Masao}, 
  journal  = {Computational Management Science}, 
  issn     = {1619-697X}, 
  doi      = {10.1007/s10287-004-0010-0}, 
  pages    = {21--56}, 
  number   = {1}, 
  volume   = {2}
}

@article{Su2004, 
  year     = {2004}, 
  title    = {A sequential {NCP} algorithm for solving equilibrium problems with equilibrium constraints}, 
  author   = {Su, Che-Lin}, 
  journal  = {Manuscript, Department of Management Science and Engineering, Stanford University, Stanford, CA}, 
  url      = {http://citeseerx.ist.psu.edu/viewdoc/download?doi=10.1.1.101.6677\&rep=rep1\&type=pdf}, 
  pages    = {1--24}
}

@article{Vazifeh2021, 
  year     = {2021}, 
  title    = {Biomass supply chain coordination for remote communities: A game-theoretic modeling and analysis approach}, 
  author   = {Vazifeh, Zahra and Mafakheri, Fereshteh and An, Chunjiang}, 
  journal  = {Sustainable Cities and Society}, 
  issn     = {2210-6707}, 
  doi      = {10.1016/j.scs.2021.102819}, 
  pages    = {102819}, 
  volume   = {69}
}

@article{Xiong2019, 
  year     = {2019}, 
  title    = {Cloud/Edge Computing Service Management in Blockchain Networks: Multi-Leader Multi-Follower Game-Based {ADMM} for Pricing}, 
  author   = {Xiong, Zehui and Kang, Jiawen and Niyato, Dusit and Wang, Ping and Poor, H. Vincent}, 
  journal  = {{IEEE} Transactions on Services Computing}, 
  issn     = {1939-1374}, 
  doi      = {10.1109/tsc.2019.2947914}, 
  pages    = {356--367}, 
  number   = {2}, 
  volume   = {13}
}

@article{Xu2021, 
  year     = {2021}, 
  title    = {Privacy-preserving incentive mechanism for multi-leader multi-follower {IoT}-edge computing market: A reinforcement learning approach}, 
  author   = {Xu, Huiying and Qiu, Xiaoyu and Zhang, Weikun and Liu, Kang and Liu, Shuo and Chen, Wuhui}, 
  journal  = {Journal of Systems Architecture}, 
  issn     = {1383-7621}, 
  doi      = {10.1016/j.sysarc.2020.101932}, 
  pages    = {101932}, 
  volume   = {114}
}

@inproceedings{Yao2024,
 author = {Yao, Wei and Yin, Haian and Zeng, Shangzhi and Zhang, Jin},
 booktitle = {International Conference on Learning Representations},
 editor = {Y. Yue and A. Garg and N. Peng and F. Sha and R. Yu},
 pages = {55516--55549},
 title = {Overcoming Lower-Level Constraints in Bilevel Optimization: A Novel Approach with Regularized Gap Functions},
 url = {https://proceedings.iclr.cc/paper_files/paper/2025/file/8b8fe72f3193fe78ac353ebcc686b395-Paper-Conference.pdf},
 volume = {2025},
 year = {2025}
}

@article{Yao2024gy6, 
  year     = {2024}, 
  title    = {Constrained Bi-Level Optimization: Proximal {L}agrangian Value function Approach and {H}essian-free Algorithm}, 
  author   = {Yao, Wei and Yu, Chengming and Zeng, Shangzhi and Zhang, Jin}, 
  journal  = {{arXiv}}, 
  doi      = {10.48550/arxiv.2401.16164}, 
  eprint   = {2401.16164}
}

@article{Ye1995, 
  year     = {1995}, 
  title    = {Optimality conditions for bilevel programming problems}, 
  author   = {Ye, J. J. and Zhu, D. L.}, 
  journal  = {Optimization}, 
  issn     = {0233-1934}, 
  doi      = {10.1080/02331939508844060}, 
  pages    = {9--27}, 
  number   = {1}, 
  volume   = {33}
}

@article{Ye1997, 
  year     = {1997}, 
  title    = {Exact Penalization and Necessary Optimality Conditions for Generalized Bilevel Programming Problems}, 
  author   = {Ye, J. J. and Zhu, D. L. and Zhu, Q. J.}, 
  journal  = {{SIAM} Journal on Optimization}, 
  issn     = {1052-6234}, 
  doi      = {10.1137/s1052623493257344}, 
  pages    = {481--507}, 
  number   = {2}, 
  volume   = {7}
}

@article{Zhang201675s, 
  year     = {2016}, 
  title    = {A Multi-Leader Multi-Follower {S}tackelberg Game for Resource Management in {LTE} Unlicensed}, 
  author   = {Zhang, Huaqing and Xiao, Yong and Cai, Lin X. and Niyato, Dusit and Song, Lingyang and Han, Zhu}, 
  journal  = {{IEEE} Transactions on Wireless Communications}, 
  issn     = {1536-1276}, 
  doi      = {10.1109/twc.2016.2623603}, 
  pages    = {348--361}, 
  number   = {1}, 
  volume   = {16}
}

\end{document}